\begin{filecontents*}{example.eps}
\RequirePackage{fix-cm}

\documentclass[smallextended]{svjour3}

\usepackage{xfrac, tikz, amsfonts, amsmath, hyperref, rotating, enumitem, schemata}
\usetikzlibrary{arrows.meta}

\usepackage{ulem}

\newcommand{\mbf}[1]{\mathbf{#1}}
\newcommand{\mbb}[1]{\mathbb{#1}}
\newcommand{\mbfs}[1]{\boldsymbol{#1}}
\newcommand{\mcf}[1]{\mathcal{#1}}

\newcommand{\fpt}{\mathcal{O}_{{\rm FPT}}}

\newcommand{\conv}{{\rm conv}}
\newcommand{\supp}{{\rm supp}}
\newcommand{\Span}{{\rm span}}
\newcommand{\floor}[1]{\lfloor#1\rfloor}

\begin{document}

\title{A Colorful Steinitz Lemma with Application to Block-Structured Integer Programs}

\titlerunning{A Colorful Steinitz Lemma with Applications to Block IPs}

\author{Timm Oertel \and Joseph Paat\and Robert Weismantel}
\authorrunning{Oertel et al.}
\institute{T. Oertel \at Department of Data Science, Friedrich-Alexander Universit\"{a}t, Erlangen-N\"{u}rnberg, Germany
\and
J. Paat \at Sauder School of Business, University of British Columbia, Vancouver BC, Canada
\and
R. Weismantel \at Department of Mathematics, Institute for Operations Research, ETH Z\"{u}rich, Switzerland}
\date{}

\maketitle

\begin{abstract}The Steinitz constant in dimension $d$ is the smallest value $c(d)$ such that for any norm on $\mbb{R}^{ d}$ and for any finite zero-sum sequence in the unit ball, the sequence can be permuted such that the norm of each partial sum is bounded by $c(d)$. 
Grinberg and Sevastyanov prove that $c(d) \le d$ and that the bound of $d$ is best possible for arbitrary norms; we refer to their result as the Steinitz Lemma.
We present a variation of the Steinitz Lemma that permutes multiple sequences at one time. 
Our result, which we term a {\it colorful Steinitz Lemma}, demonstrates upper bounds that are independent of the number of sequences. 

Many results in the theory of integer programming are proved by permuting vectors of bounded norm; this includes proximity results, Graver basis algorithms, and dynamic programs.
Due to a recent paper of Eisenbrand and Weismantel, there has been a surge of research on how the Steinitz Lemma can be used to improve integer programming results. 
As an application we prove a proximity result for block-structured integer programs. 


\end{abstract}



\section{Introduction.}

Let $\|\cdot\|: \mbb{R}^d \to \mbb{R}$ be an arbitrary norm with corresponding unit ball 
\[
\mcf{U} := \{ \mbf{x} \in \mbb{R}^d:\ \|\mbf{x}\|\le 1\}.
\]
A sequence $(\mbf{u}^i)_{i=1}^m$ in $\mbb{R}^d$ is a {\bf zero-sum sequence} if $\sum_{i=1}^m \mbf{u}^i = \mbf{0}$.
Grinberg and Sevastyanov prove the following result on zero-sum sequences.
We refer to Theorem~\ref{thmSteinitz} as the {\it Steinitz Lemma} because Steinitz~\cite{S1913} originally proves the result, albeit with a larger upper bound. 

\begin{theorem}[Grinberg and Sevastyanov~\cite{GS1980}]\label{thmSteinitz}
Let $\|\cdot\|: \mbb{R}^d \to \mbb{R}$ be a norm with unit ball $\mcf{U}$.
For every zero-sum sequence $(\mbf{u}^i)_{i=1}^m$ in $\mcf{U}$, there exists a permutation $\pi \in \mcf{S}^m$ such that 
\[
\left\|\sum_{i=1}^k \mbf{u}^{\pi(i)} \right\| \le d
\] 
for each $k \in \{ 1, \ldots, m\}$.
\end{theorem}
The Steinitz Lemma permutes a single zero-sum sequence $(\mbf{u}^i)_{i=1}^m$ in $\mcf{U}$.
In this paper, we consider permuting multiple sequences $(\mbf{u}^{i}_1)_{i=1}^m, \ldots, (\mbf{u}^{i}_n)_{i=1}^m$ in $\mcf{U}$ whose union $(\mbf{u}^{i}_j)_{i,j}$ is a zero sum sequence.
%
Theorem~\ref{thmSteinitz} guarantees a permutation $\pi \in \mcf{S}^{nm}$ on $(\mbf{v}^k)_{k=1}^{nm} := (\mbf{u}^{i}_j)_{i,j}$ such that each partial sum $\sum_{i=1}^k \mbf{v}^{\pi(i)}$ has a bounded norm. 
The permuted sequence $(\mbf{v}^{\pi(k)})_{k=1}^{nm}$ may mix the original sequences $(\mbf{u}^{i}_1)_{i=1}^m, \ldots, (\mbf{u}^{i}_n)_{i=1}^m$ arbitrarily.
We are interested in permutations that equally distribute the vectors from the original sequences.
In particular, we are interested in permutations $\pi_1, \ldots, \pi_n \in \mcf{S}^m$ such that
\begin{equation}\label{eqBalancedSum}
\left\|\sum_{i=1}^k \sum_{j=1}^n \mbf{u}^{\pi_j(i)}_j\right\|
\end{equation}
has bounded norm independent of $m$ and $n$ for each $k \in \{1, \ldots, m\}$.
The partial sum in~\eqref{eqBalancedSum} has exactly $k$ elements from each of the original sequences $(\mbf{u}^{i}_1)_{i=1}^m, \ldots, (\mbf{u}^{i}_n)_{i=1}^m$.
Figure~\ref{fig:ColorfulSteinitz} illustrates the type of permutations that we consider.

\begin{center}
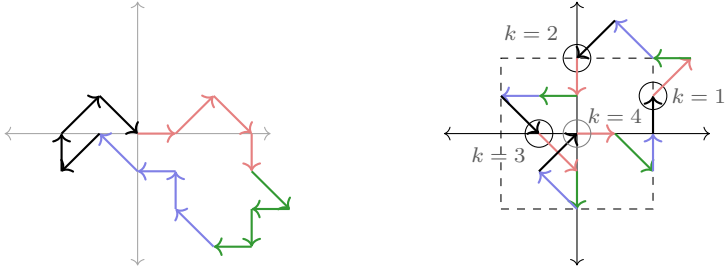
\begin{figure}
\begin{center}
\begin{tabular}{c@{\hskip 2 cm}c}
\begin{tikzpicture}[scale = .5]
\draw[<->, color = black!35](-3.5,0)--(3.5,0);
\draw[<->,color = black!35](0,-3.5)--(0,3.5);

\draw[thick, ->, draw = red!80!black!50](0,0)--(1,0);
\draw[thick, ->, draw = red!80!black!50](1,0)--(2,1);
\draw[thick, ->, draw = red!80!black!50](2,1)--(3,0);
\draw[thick, ->, draw = red!80!black!50](3,0)--(3,-1);

\draw[thick, ->, draw = green!50!black!80](3,-1)--(4,-2);
\draw[thick, ->, draw = green!50!black!80](4,-2)--(3,-2);
\draw[thick, ->, draw = green!50!black!80](3,-2)--(3,-3);
\draw[thick, ->, draw = green!50!black!80](3,-3)--(2,-3);

\draw[thick, ->, draw = blue!80!black!50](2,-3)--(1,-2);
\draw[thick, ->, draw = blue!80!black!50](1,-2)--(1,-1);
\draw[thick, ->, draw = blue!80!black!50](1,-1)--(0,-1);
\draw[thick, ->, draw = blue!80!black!50](0,-1)--(-1,0);

\draw[thick, ->, draw = black](-1,0)--(-2,-1);
\draw[thick, ->, draw = black](-2,-1)--(-2,0);
\draw[thick, ->, draw = black](-2,0)--(-1,1);
\draw[thick, ->, draw = black](-1,1)--(0,0);

\end{tikzpicture}
&
\begin{tikzpicture}[scale = .5]
\draw[<->, color = black](-3.5,0)--(3.5,0);
\draw[<->, color = black](0,-3.5)--(0,3.5);

\draw[dashed, draw = black] (-2,-2)--(2,-2)--(2,2)--(-2,2)--cycle;

\draw[thick, ->, draw = red!80!black!50](0,0)--(1,0);
\draw[thick, ->, draw = green!50!black!80](1,0)--(2,-1);
\draw[thick, ->, draw = blue!80!black!50](2,-1)--(2,0);
\draw[thick, ->, draw = black](2,0)--(2,1);
\draw[draw = black, fill = none] (2,1) circle (3 ex) node[ right = 4]{\color{black!65}\footnotesize $k=1$}; 

\draw[thick, ->, draw = red!80!black!50](2,1)--(3,2);
\draw[thick, ->, draw = green!50!black!80](3,2)--(2,2);
\draw[thick, ->, draw = blue!80!black!50](2,2)--(1,3);
\draw[thick, ->, draw = black](1,3)--(0,2);
\draw[draw = black, fill = none] (0,2) circle (3 ex) node[ above left = 4]{\color{black!65}\footnotesize $k=2$}; 

\draw[thick, ->, draw = red!80!black!50](0,2)--(0,1);
\draw[thick, ->, draw = green!50!black!80](0,1)--(-1,1);
\draw[thick, ->, draw = blue!80!black!50](-1,1)--(-2,1);
\draw[thick, ->, draw = black](-2,1)--(-1,0);
\draw[draw = black, fill = none] (-1,0) circle (3 ex) node[below left = 2]{\color{black!65}\footnotesize $k=3$}; 

\draw[thick, ->, draw = red!80!black!50](-1,0)--(0,-1);
\draw[thick, ->, draw = green!50!black!80](0,-1)--(0,-2);
\draw[thick, ->, draw = blue!80!black!50](0,-2)--(-1,-1);
\draw[thick, ->, draw = black](-1,-1)--(0,0);
\draw[draw = black!50, fill = none] (0,0) circle (3 ex) node[above right = 1]{\color{black!65}\footnotesize $k=4$}; 

\end{tikzpicture}
\end{tabular}
\end{center}
\caption{On the left we have four sequences, each containing four vectors.
Each sequence is drawn in a different color, and their union has zero-sum.
On the right we permute each sequence individually such that the sum of the $4k$ vectors consisting of the first $k$ permuted vectors in each sequence is bounded for each $k \in \{1, 2,3,4\}$; these four sums are highlighted with circles.
This bound is uniform over $k$, and the bounding box is drawn in gray.
}\label{fig:ColorfulSteinitz}
\end{figure}
\end{center}

One can upper bound~\eqref{eqBalancedSum} using the Stenitz Lemma.
For example, the sequence $(\mbf{v}^i)_{i=1}^m$, where $\mbf{v}^i := \sum_{j=1}^n \mbf{u}^{i}_j$, lies in $n \cdot \mcf{U}$.
Applying Theorem~\ref{thmSteinitz} yields a permutation $\pi \in \mcf{S}^m$ such that
\begin{equation}\label{eqTrivialBound}
\left\|\sum_{i=1}^k \mbf{v}^{\pi(i)}\right\|= \left\|\sum_{i=1}^k \sum_{j=1}^n \mbf{u}^{\pi(i)}_j\right\| \le nd
\end{equation}
for all $k \in \{ 1, \ldots, m\}$.
Setting $\pi_1 = \cdots = \pi_n := \pi$ yields the upper bound of $nd$, which depends on the number of sequences $n$. 
Our main result is the existence of $\pi_1, \ldots, \pi_n$ that upper bound~\eqref{thmSteinitz} independently of $n$.

\begin{theorem}[Colorful Steinitz Lemma]\label{thmVarSteinitz}
Let $\|\cdot\|: \mbb{R}^d \to \mbb{R}$ be a norm with unit ball $\mcf{U}$.
Let $(\mbf{u}^{i}_1)_{i=1}^m, \ldots, (\mbf{u}^{i}_n)_{i=1}^m$ be sequences in $\mcf{U}$ whose union $(\mbf{u}^{i}_j)_{i,j}$ is a zero-sum sequence.
There exist permutations $\pi_1,\ldots, \pi_n \in \mathcal{S}^m$ such that 
\[
\left\| \sum_{i=1}^k \sum_{j=1}^n \mbf{u}^{\pi_j(i)}_j \right\| \le \min\left\{nd, 40d^5\right\}
\]
for each $k \in \{ 1, \ldots, m\}$.
\end{theorem}

One may consider Theorem~\ref{thmVarSteinitz} as a variation of the classical Steinitz Lemma where we require the permutation $\pi$ to be from a significantly smaller subset  of $\mcf{S}^{mn}$.
Alternatively, it may be viewed as a colorful variation of the Steinitz Lemma that is independent of the number of colors.
The term `colorful' is borrowed from other named results that generalize classic results  in discrete geometry from one to multiple sets, e.g., the colorful Carath\'{e}odory Theorem and the colorful Helly Theorem~\cite{ADS2017,B1982}.
Chen et al. derive a colorful version of the Steinitz Lemma~\cite[Lemma 9]{CCZ2021}, but unlike Theorem~\ref{thmVarSteinitz} their upper bound depends on the number of sequences. 

The following result, which is a direct corollary of Theorem~\ref{thmVarSteinitz}, considers multiple sequences whose union does not necessarily have zero-sum.
\begin{corollary}\label{corAffineVarSteinitz}
Let $\|\cdot\|: \mbb{R}^d \to \mbb{R}$ be a norm with unit ball $\mcf{U}$.
Let $(\mbf{u}^{i}_1)_{i=1}^m, \ldots,$ $(\mbf{u}^{i}_n)_{i=1}^m$ be sequences in $\mcf{U}$.
There exist permutations $\pi_1, \ldots, \pi_n \in \mathcal{S}^m$ such that
\[
\left\| \sum_{i=1}^k \sum_{j=1}^n \mbf{u}^{\pi_j(i)}_j  - \frac{k}{m} \cdot \sum_{i=1}^m \sum_{j=1}^n \mbf{u}^{i}_j\right\| 
\le \min\left\{nd, 40d^5\right\}
\]
for each $k \in\{1, \ldots, m\}$. 
\end{corollary}

Theorems~\ref{thmSteinitz} and~\ref{thmVarSteinitz} require $m$ different partial sums to have bounded norm. 
The upper bounds in those results can be improved if we only care for a single partial sum to have bounded norm. 
Ambrus et al.~\cite{ABG2016} bound a single partial sum in the setting of the classic Steinitz Lemma.
Our next theorem extends Ambrus et al.'s result to the colorful setting.
\begin{theorem}\label{thVarSteinitz}
Let $\|\cdot\|: \mbb{R}^d \to \mbb{R}$ be a norm with unit ball $\mcf{U}$.
Let $(\mbf{u}^{i}_1)_{i=1}^m, \ldots,$ $(\mbf{u}^{i}_n)_{i=1}^m$ be sequences in $\mcf{U}$ whose union $(\mbf{u}^{i}_j)_{i,j}$ is a zero-sum sequence.
For each $k \in \{1,\ldots,m\}$, there exist $I_1, \ldots, I_n\subseteq\{1,\ldots,m\}$ such that $|I_1| = \cdots = |I_n| =k$ and 
\[
\left\| \sum_{j=1}^n\sum_{i \in I_j}  \mbf{u}^i_j \right\| \le d.
\]
\end{theorem}
%

%

The question of bounding a single partial sum bears similarities to another variation of the Steinitz Lemma in which one permutes a (not necessarily zero-sum) sequence such that some partial sum lies in $\mcf{U}$~\cite{BMMP2012,DFG2012}.

We prove Theorems~\ref{thmVarSteinitz} and~\ref{thVarSteinitz} in Section~\ref{secSteinitz}.
%

\subsection{An Application of Theorem~\ref{thmVarSteinitz} to Block Integer Programs.}\

Following the work of Eisenbrand and Weismantel~\cite{EW2018}, the Steinitz Lemma has been used in numerous projects; we point to~\cite{CKXS2019,CEHRW2020,EHK2018,EHKKLO2019,JR2018,K2020} just to name a few.
One area that benefits from the Steinitz Lemma is the study of integer programs with special sparse block-structures. 
%
%
Block-structured integer programs can be applied in various problems such as scheduling and social choice; see, e.g.,~\cite{JKMR2021,KKM2020,SSV1996}.

A particular family of block-structured integer programs is defined by the {\it 4-block matrix}

\begin{equation}\label{eq3BlockMatrix}
\mbf{H} := 
\left[
\begin{array}{c|c@{\hskip .15 cm}c@{\hskip .15 cm}c@{\hskip .15 cm}c}
\mbf{A}^0 & \mbf{C}^1 & \cdots & \mbf{C}^n\\
 \hline &\\[-.35 cm]
 \mbf{B}^1 & \mbf{A}^1 \\
 \vdots & & \ddots\\
 \mbf{B}^n & &&\mbf{A}^n
\end{array}
\right] \in \mbb{Z}^{(s_0+ns) \times (t_0+nt)},
\end{equation}
where $\mbf{B}^i \in \mbb{Z}^{s \times t_0}$, $\mbf{A}^i \in \mbb{Z}^{s \times t}$, $\mbf{C}^i \in \mbb{Z}^{s_0 \times t}$ for each $i \in \{ 1, \ldots, n\}$, and $\mbf{A}^0 \in \mbb{Z}^{s_0\times t_0}$.
Our results hold if $s$ and $t$ depend on $i$, but we omit this dependence for the sake of presentation.
Given a righthand side $\mbf{b} = (\mbf{b}^0, \mbf{b}^1, \ldots, \mbf{b}^n)$\footnote{It will be helpful to write column vectors inline. For $\mbf{d}^1\in \mbb{R}^d$ and $\mbf{d}^2 \in \mbb{R}^{d'}$, we use $ (\mbf{d}^1, \mbf{d}^2)$ to denote the column vector $[(\mbf{d}^1)^\top, (\mbf{d}^2)^\top]^\top$.} with $\mbf{b}^0 \in \mbb{Z}^{s_0}$ and $\mbf{b}^1, \ldots, \mbf{b}^n \in \mbb{Z}^{s}$, an objective vector $(\mbf{c}^{\mbf{x}}, \mbf{c}^{\mbf{y}})  \in \mbb{R}^{t_0} \times \mbb{R}^{nt}$, and upper bounds $\mbf{u}^{\mbf{x}} \in (\mbb{Z}_+ \cup \{\infty\})^{t_0}$ and $\mbf{u}^{\mbf{y}} \in (\mbb{Z}_+ \cup \{\infty\})^{nt}$, the {\it 4-block integer program} is
\begin{equation}\label{eqMainProb}
\max \left\{ \mbf{c}^{\mbf{x}} \cdot \mbf{x} + \mbf{c}^{\mbf{y}} \cdot \mbf{y} :\ 
\left[\begin{array}{c}\mbf{x}\\ \mbf{y}\end{array}\right] \in \mbb{Z}^{t_0}_+ \times \mbb{Z}^{nt}_+,\
\mbf{H} \left[\begin{array}{c}\mbf{x}\\ \mbf{y}\end{array}\right] = \mbf{b},\
\left[\begin{array}{c}\mbf{x}\\ \mbf{y}\end{array}\right]\le\left[\begin{array}{c}\mbf{u}^{\mbf{x}}\\ \mbf{u}^{\mbf{y}}\end{array}\right]
%
 \right\}.
\end{equation}
The family of $4$-block integer programs generalize $n$-fold integer programs, which occur when $t_0 =0$, and $2$-stage stochastic integer programs, which occur when $s_0 =0 $; see~\cite{CEHRW2020,HDOW2008,HOR2013,KKM2020b} and~\cite{K2020,KLO2018}.
In our results, we assume $t_0, s_0 \ge 1$.

We apply the colorful Steinitz Lemma to study orthant-compatible vectors in $\ker\mbf{H}$. 
For vectors $\mbf{x} = (x_i), \mbf{y} = (y_i) \in \mbb{R}^d$, we write $\mbf{x} \sqsubseteq \mbf{y}$ if $|x_i| \le |y_i|$ and $|x_i| \cdot |y_i| \ge 0$ for each $i \in \{1, \dotsc, d\}$. 
Orthant-compatible vectors in $\ker\mbf{H}$ appear in the study of Graver bases and proximity.

The {\it Graver basis} of $\mbf{H}$ is the set of all nonzero vectors $\mbf{z} \in \mbb{Z}^{t_0+nt} \cap \ker \mbf{H}$ such that there does not exist a different nonzero vector $\mbf{u} \in \mbb{Z}^{t_0+nt} \cap \ker \mbf{H}$ satisfying $\mbf{u} \sqsubseteq \mbf{z}$.
Graver bases can be used to solve block-structured integer programs~\cite{CKXS2019,EHKKLO2019,HDOW2008,HKW2010,K2020}.

The {\it proximity} problem can be stated as follows:
Given a norm $\|\cdot\|$ and an optimal vertex solution $(\widehat{\mbf{x}}, \widehat{\mbf{y}})$ to the linear relaxation of~\eqref{eqMainProb}, upper bound the distance $\|(\widehat{\mbf{x}}, \widehat{\mbf{y}}) - (\overline{\mbf{x}}, \overline{\mbf{y}})\|$ to the nearest optimal integer solution $(\overline{\mbf{x}}, \overline{\mbf{y}})$ (if any exist). 
Proximity results are used in the analysis of integer programming algorithms, including to limit the state space of dynamic programs and to bound the integrality gap.
Many proximity proofs rely on the following fact: there does not exist a nonzero vector $(\mbf{u}, \mbf{v}) \in (\mbb{Z}^{t_0}\times \mbb{Z}^{nt}) \cap \ker \mbf{H}$ satisfying $(\mbf{u}, \mbf{v}) \sqsubseteq (\widehat{\mbf{x}} - \overline{\mbf{x}}, \widehat{\mbf{y}} - \overline{\mbf{y}})$ unless $(\widehat{\mbf{x}} - \overline{\mbf{x}}, \widehat{\mbf{y}} - \overline{\mbf{y}})$ is already integer-valued.
For more on this fact, see the discussion of {\it cycles} in~\cite{EW2018}.

Graver bases and the proximity problem consider vectors $\mbf{g} \in \ker\mbf{H}$ that have no nonzero integer vectors $\mbf{h} \in \ker\mbf{H}$ satisfying $\mbf{h} \sqsubseteq \mbf{g}$.
We apply the colorful Steinitz Lemma in Theorem~\ref{thm4BlockVector} to bound the size of such vectors $\mbf{g}$.
When analyzing block-structured integer programs, one is often interested in how complex the problem is as $n$ grows large. 
In this framework, the variables $s,t,s_0,$ and $t_0$ are commonly considered fixed values. 
Also, the {\it largest absolute entry $\Delta$} in the constraint matrix~\eqref{eq3BlockMatrix} is considered fixed.
We adopt this fixed parameter convention; for a function $f(n,s,t,s_0,t_0,\Delta)$ we write $f(n,s,t,s_0,t_0,\Delta)\in \fpt(n^\alpha)$ if $f$ can be upper bounded by a function $n^\alpha \cdot  g(s,t,s_0,t_0,\Delta)$.

\begin{theorem}\label{thm4BlockVector}
There exists a number $\xi\in \fpt(n^{\min \{t_0+2, s_0\}})$ such that the following holds: 
For each $(\widehat{\mbf{x}}, \widehat{\mbf{y}}) \in (\mbb{R}_+^{t_0} \times \mbb{R}_+^{nt})\cap  \ker \mbf{H}$ such that $\|(\widehat{\mbf{x}},\widehat{\mbf{y}})\|_{\infty} > \xi$, there exists a nonzero $({\mbf{x}}, {\mbf{y}}) \in (\mbb{Z}^{t_0}_+ \times \mbb{Z}^{nt}_+) \cap  \ker \mbf{H}$ such that $ ({\mbf{x}}, {\mbf{y}}) \le (\widehat{\mbf{x}}, \widehat{\mbf{y}})$.
\end{theorem}

We emphasize that $(\widehat{\mbf{x}}, \widehat{\mbf{y}})$ in the statement of Theorem~\ref{thm4BlockVector} is not necessarily integer-valued.
The assumption $(\widehat{\mbf{x}}, \widehat{\mbf{y}}) \in \mbb{R}_+^{t_0} \times \mbb{R}_+^{nt}$ in Theorem~\ref{thm4BlockVector} (as opposed to $(\widehat{\mbf{x}}, \widehat{\mbf{y}})$ living in some other orthant of $\mbb{R}^{t_0} \times \mbb{R}^{nt}$) is made without loss of generality by multiplying columns of $\mbf{H}$ by $-1$.
One value of $\xi$ that satisfies Theorem~\ref{thm4BlockVector} is defined in~\eqref{eqLargeConstant}.
Theorem 2 in~\cite{CKXS2019} provides a lower bound example illustrating that Theorem~\ref{thm4BlockVector} is nearly optimal in this $\fpt$ framework.
Their example has $t_0, s_0 \in O(k)$ for some $k \in \mbb{Z}_+$  and $\|\widehat{\mbf{x}}\|_{\infty}\in \Omega(n^{k})$.

We use Theorem~\ref{thm4BlockVector} to derive new bounds on Graver basis elements for $4$-block integer programs.
For $n$-fold matrices, a Graver basis element $\mbf{g}$ satisfies $\|\mbf{g}\|_1 \in \mcf{O}(s_0s\Delta)^{(s_0+1)(s+1)}$~\cite[ii) on page 49:4]{EHK2018}.
For $2$-stage stochastic matrices, a Graver basis element $\mbf{g}$ satisfies $\|\mbf{g}\|_{\infty} \le (st_0\Delta)^{\mcf{O}(st_0(2s\Delta+1)^{st_0^2})}$~\cite[Theorem 2]{K2020}.
For the $4$-block matrix $\mbf{H}$, Chen et al. prove that a Graver basis element $\mbf{g}$ satisfies $\|\mbf{g}\|_{\infty} \in \fpt(n^{s_0})$~\cite[Theorem 1]{CKXS2019} and $\|\mbf{g}\|_{\infty} \in \fpt(n^{t^2+1})$ if $\mbf{A}^0 = \mbf{0}$~\cite[Theorem 5]{CKXS2019}.
We use Theorem~\ref{thm4BlockVector} to replace $s_0$ in the exponent of Chen et al.'s general bound with $\min \{t_0+2, s_0\}$.
In the $\fpt$ framework, our bound matches the one in~\cite{CKXS2019} and provides an improvement when $t_0$ is bounded.

\begin{corollary}\label{corGraverBasis}
A Graver basis element $\mbf{g}$ of $\mbf{H}$ satisfies 
\[
\|\mbf{g}\|_{\infty} \in \fpt(n^{\min \{t_0+2, s_0\}}).
\]
\end{corollary}
Chen et al. propose an algorithm for $4$-block integer programs that runs in time $\fpt(n^{O(s_0t_0)})$~\cite[Theorem 3]{CKXS2019} and $\min\{\fpt(n^{\mathcal{O}(s_0t)}), \fpt(n^{\mathcal{O}(t^2t_0)})\}$ if $\mbf{A}^0 = \mbf{0}$~\cite[Theorem 7]{CKXS2019}.
Cslovjecsek et al.~\cite{CEHRW2020} take a different approach to solving $n$-fold integer programs; rather than following an augmentation scheme, they instead solve an appropriate mixed integer relaxation that satisfies a stronger proximity result.
In a similar way, Corollary~\ref{corMain} provides an algorithm for the $4$-block integer program:
\begin{enumerate}[label= Step \alph*,leftmargin=*,align=left]
\item[Step 1.] Compute an optimal solution $(\widehat{\mbf{x}}, \widehat{\mbf{y}}) \in \mbb{R}^{t_0}_+ \times \mbb{R}^{nt}_+$ to the linear relaxation of~\eqref{eqMainProb}.
\item[Step 2.] Enumerate the $\fpt(n^{t_0 \cdot \min \{t_0+2, s_0\}})$ integer vectors $\overline{\mbf{x}} \in \mbb{Z}^{t_0}_+$ such that $\|\widehat{\mbf{x}} - \overline{\mbf{x}}\|_{\infty} \le \fpt(n^{\min \{t_0+2, s_0\}})$.
\item[Step 3.] For each $\overline{\mbf{x}}$ enumerated in Step 2, find an optimal solution $\overline{\mbf{y}}$ to the $n$-fold integer program $\max\{\mbf{c}^{\mbf{y}} \cdot \mbf{y} : \widetilde{\mbf{H}} \mbf{y} = \mbf{b} - \mbf{B} \overline{\mbf{x}}\}$, where
\[
\widetilde{\mbf{H}} := 
\left[
\begin{array}{cccc}
 \mbf{C}^1 & \cdots & \mbf{C}^n\\
 \hline \\[-.35 cm]
 \mbf{A}^1 \\
 & \ddots\\
 &&\mbf{A}^n
\end{array}
\right]~\text{and}~
\mbf{B} := 
\left[
\begin{array}{cccc}
 \mbf{A}^0\\
 \hline\\[-.35 cm]
\mbf{B}^1\\
\vdots\\
\mbf{B}^n
\end{array}
\right].
\]
\item[Step 4.] From the solutions $\overline{\mbf{y}}$ in Step 3, return $(\overline{\mbf{x}}, \overline{\mbf{y}})$ maximizing $\mbf{c}^{\mbf{x}} \cdot \overline{\mbf{x}} + \mbf{c}^{\mbf{y}}\cdot \overline{\mbf{y}}$. 
\end{enumerate}
\begin{corollary}\label{corRunTime}
Steps~1 to 4 show that a $4$-block integer program can be solved in time $\fpt(X+ Yn^{t_0 \cdot\min \{t_0+2, s_0\}})$, where $X$ is the time to solve the linear relaxation of~\eqref{eqMainProb} and $Y$ is the time to solve an $n$-fold integer program.
\end{corollary}

We can also use Theorem~\ref{thm4BlockVector} to derive new bounds on the proximity problem for $4$-block integer programs.
Cook et al.~\cite{CGST1986} establish one of the first $\ell_{\infty}$-proximity bounds for general integer programs.
Since their work, a variety of bounds have been established~\cite{AHO2020,EW2018,LPSX2020,LPSX2021,PWW2018}.
Notable among these is the result of Eisenbrand and Weismantel~\cite{EW2018}, who use Theorem~\ref{thmSteinitz} to bound the $\ell_1$-proximity by a function of $\Delta$ and the number of equations (in particular, it is independent of the dimension $n$).
However, even the strongest of these general results are polynomial in the dimensions and the largest minor of the constraint matrix; see~\cite{CKPW2021} for $\ell_\infty$-proximity and~\cite{EW2018,LPSX2021} for $\ell_1$-proximity.
Applying these general results to~\eqref{eqMainProb} would yield a proximity bound larger than $\Delta^n$, which is the order of the largest matrix minor. 
For $n$-fold integer programs, Cslovjecsek et al.~\cite[Theorem 4.3]{CEHRW2020} demonstrate a bound of $(2s_0\Delta G+1)^{s_0+4}$ on $\ell_1$-proximity, where  $G \in \mcf{O}(s_0s\Delta)^{(s_0+1)(s+1)}$ is a bound on the height of the Graver basis elements of an $n$-fold integer program~\cite[Lemma 3]{EHK2018}.
For 2-stage stochastic integer programs, Cslovjecsek et al.~\cite[Lemma 4]{CEPVW2021} demonstrate a bound of $2^{\mcf{O}(t_0(t+t_0)\Delta)^{t_0(t_0+t)}}$ on $\ell_{\infty}$-proximity.
We use Theorem~\ref{thm4BlockVector} to bound the $\ell_{\infty}$-proximity for $4$-block integer programs.

\begin{corollary}\label{corMain}
Let $(\widehat{\mbf{x}}, \widehat{\mbf{y}}) \in \mbb{R}^{t_0}_+ \times \mbb{R}^{nt}_+$ be an optimal solution to the linear relaxation of~\eqref{eqMainProb}.
If~\eqref{eqMainProb} is feasible, then it is has an optimal solution $(\overline{\mbf{x}}, \overline{\mbf{y}})$ such that 
\[
\|(\widehat{\mbf{x}}, \widehat{\mbf{y}})  - (\overline{\mbf{x}}, \overline{\mbf{y}}) \|_{\infty} \in \fpt\left(n^{\min \{t_0+2, s_0\}}\right).
\]
\end{corollary}

\noindent Corollary~\ref{corMain} follows from Theorem~\ref{thm4BlockVector} by using the standard cycle argument in~\cite{EW2018}.

\section{A Colorful Version of the Steinitz Lemma.}\label{secSteinitz}

\begin{proof}[of Theorem~\ref{thmVarSteinitz}]
Throughout the proof, the index $j \in \{1, \dotsc, n\}$ refers to the different `colors' while the index $i$ refers to the $i$th vector in a particular color.
Inequality~\eqref{eqTrivialBound} proves the upper bound of $nd$.
The limiting factor of this approach is that $\left\| \sum_{j=1}^n  \mbf{u}^i_j \right\|$ can grow linearly in $n$. 
To prove the upper bound of $40d^5$, the core idea is to show that there exists permutations $\sigma_1,\ldots,\sigma_n \in \mcf{S}^m$ such that $\left\| \sum_{j=1}^n  \mbf{u}^{\sigma_j(i)}_j \right\|$ can be bounded independently of the number of colors $n$ for all $i\in\{1,\ldots,m\}$.
For that let $\sigma_1,\ldots,\sigma_n \in \mcf{S}^m$ be permutations that minimize 
\begin{equation}\label{eqSteinitzMin}
\max_{i\in \{1,\ldots,m\}}\ \left\| \sum_{j=1}^n  \mbf{u}^{\sigma_j(i)}_j \right\|.
\end{equation}
Furthermore, choose $\sigma_1,\ldots,\sigma_n$ in such a way that they minimize the number of indices $i$ for which the maximum in~\eqref{eqSteinitzMin} is attained. 
We claim
\begin{equation}
\left\| \sum_{j=1}^n \mbf{u}^{\sigma_j(i)}_j \right\| \le (d+1)^2(4d(d+1)+2) \label{eqSteinitzClaim}
\end{equation}
for all $i\in \{1,\ldots,m\}$.
After we establish~\eqref{eqSteinitzClaim}, we can apply the classic Theorem~\ref{thmSteinitz} to the vectors $\sum_{j=1}^n \mbf{u}^{\sigma_j(1)}_j, \dotsc, \sum_{j=1}^n \mbf{u}^{\sigma_j(m)}_j$ to prove the existence of a permutation $\rho \in \mathcal{S}^m$ such that
\[
\left\|\sum_{i=1}^k \sum_{j=1}^n \mbf{u}^{\sigma_j({\rho(i)})}_j \right\|
\le d \left((d+1)^2(4d(d+1)+2)\right)
\le 40d^5
\]
for each $ k \in \{1, \ldots, m\}$.
Theorem~\ref{thmVarSteinitz} will then follow by setting $\pi_j =\sigma_j\circ\rho$ for each $j \in \{ 1, \ldots, n\}$.
It remains to prove~\eqref{eqSteinitzClaim}.
We assume without loss of generality that each $\sigma_j$ is the identity.
For each index $i \in \{1, \dotsc, m\}$, we denote the sum of the $i$th vector across all colors by
\begin{equation}\label{eqAggregate1}
\mbf{u}^{i} := \sum_{j=1}^n \mbf{u}^{i}_j
\end{equation}
In order to derive a contradiction, assume that~\eqref{eqSteinitzClaim} is false.
Suppose 
\begin{equation}\label{eqTOObig}
\left\| \mbf{u}^1 \right\| = \max_{i\in \{1,\ldots,m\}} \left\| \mbf{u}^i \right\| > (d+1)^2(4d(d+1)+2).
\end{equation}

In what follows, the strategy will be to show that there exists an additional $d$ vectors, we will say $\mbf{u}^{2},\ldots,\mbf{u}^{d+1}$, such that the center of $\mbf{u}^{1},\ldots,\mbf{u}^{d+1}$ is close to the origin.
Then, by permuting within each color only the indices $\{1,\ldots,d+1\}$ and adding the $i = 1, \dotsc, d+1$ vector across all colors, we get new vectors $\overline{\mbf{u}}^{1}, \dotsc, \overline{\mbf{u}}^{d+1}$, which we show satisfy $\|\overline{\mbf{u}}^{1}\|,\ldots,\|\overline{\mbf{u}}^{d+1}\|<\|\mbf{u}^{1}\|$.
But this will contradict~\eqref{eqSteinitzMin} because $\|\mbf{u}^1\|$ satisfies~\eqref{eqTOObig}.
Note that $(\mbf{u}^i)_{i=1}^m$ is a zero-sum sequence.
By Carath\'eodory's Theorem there exist $d$ vectors, say $\mbf{u}^2,\ldots,\mbf{u}^{d+1}$, such that
\(
\mbf{0}\in\conv\{\mbf{u}^1,\ldots,\mbf{u}^{d+1}\}.
\)
Hence, there exist $\lambda_1,\ldots,\lambda_{d+1}\in \mbb{R}_+$ such that $\sum_{i=1}^{d+1}\lambda_i=1$ and $\sum_{i=1}^{d+1}\lambda_i \mbf{u}^i=\mbf{0}$.
Let $i'\in \{1, \ldots, d+1\}$ be such that $\lambda_{i'} = \max\{ \lambda_i:\ i\in \{1, \ldots, d+1\}\}$. 

Denote the center of $\mbf{u}^1, \dotsc, \mbf{u}^{d+1}$ by
\begin{equation}\label{eqCentroid1}
\mbf{c}:=\frac{1}{d+1} \cdot \sum_{i=1}^{d+1} \mbf{u}^i = \frac{1}{d+1} \cdot \sum_{i=1}^{d+1} \sum_{j=1}^n \mbf{u}^i_j.
\end{equation}
We have
\begin{align}\label{eqBoundC}
\|\mbf{c}\|=\|\mbf{c}- \mbf{0}\| & =\left\| \sum_{i=1}^{d+1} \left(\frac{1}{d+1}-\frac{\lambda_i}{\lambda_{i'}(d+1)}\right) \mbf{u}^i \right\|\nonumber\\
& \le \sum_{i=1}^{d+1} \left(\frac{1}{d+1}-\frac{\lambda_i}{\lambda_{i'}(d+1)}\right)\cdot \left\|\mbf{u}^i\right\|\nonumber\\
&\le\frac{d}{d+1}\cdot\left\|\mbf{u}^1\right\|.
\end{align}

Let $\|\cdot\|_e : \mbb{R}^{d(d+1)}\to \mbb{R}$ denote the extended norm $\|(\mbf{x}^1,\ldots,\mbf{x}^{d+1})\|_e:=\max\{\|\mbf{x}^1\|,\ldots,\|\mbf{x}^{d+1}\|\}$.
Consider the zero-sum sequence in $\mbb{R}^{d(d+1)}$
\[
\left(\left[\begin{array}{ccl}
\mbf{u}^{1}_{j} &-&\frac{1}{n}\cdot \mbf{u}^1\\
&\vdots&\\
\mbf{u}^{d+1}_{j} &-& \frac{1}{n}\cdot \mbf{u}^{d+1}
\end{array}
\right]\right)_{j=1}^n.
\]
As  $\| \mbf{u}^{i}_j - \frac{1}{n}  \mbf{u}^i\| \le 2$ for all $i\in\{1,\ldots,m\}$ and $j\in\{1,\ldots,n\}$, each vector in the sequence has norm at most $2$.
Applying Theorem~\ref{thmSteinitz} in $\mbb{R}^{d(d+1)}$ we can conclude that there exists a permutation $\tau \in \mathcal{S}^n$ such that
\begin{equation}\label{eqBounded+1sum}
\left\| \sum_{j=1}^k 
\left[\begin{array}{rcl}
\mbf{u}^{1}_{\tau(j)} &-&\frac{1}{n}\cdot \mbf{u}^1\\
&\vdots&\\
\mbf{u}^{d+1}_{\tau(j)} &-& \frac{1}{n}\cdot \mbf{u}^{d+1}
\end{array}
\right] 
\right\|_e
=
\left\| \sum_{j=1}^k 
\begin{bmatrix}
\mbf{u}^{1}_{\tau(j)}\\
\vdots\\
\mbf{u}^{d+1}_{\tau(j)}
\end{bmatrix} - 
\frac{k}{n}
 \begin{bmatrix}
 \mbf{u}^1\\
 \vdots 
\\
\mbf{u}^{d+1}
\end{bmatrix}
\right\|_e\le 2d(d+1)
\end{equation}
for each $k \in \{1, \ldots, n\}$.
For ease of presentation, we can assume without loss of generality that $\tau$ is the identity.
Note that this does not conflict with our previous assumption that $\sigma_1, \dotsc, \sigma_n$ are all equal to the identity because the former permuted within each color whereas the latter just redefines the colors.
Inequality~\eqref{eqBounded+1sum} implies that for each $i\in\{1, \ldots, d+1\}$ and any $t$ consecutive values (modulo $n$) $j_1,\ldots,j_t$ in $\{1, \dotsc, n\}$, i.e., $j_k\equiv j_{1}+k-1 \pmod n$, we have
\begin{equation*}\label{eqConsecutiveColors}
\left\|\sum_{k=1}^{t} \mbf{u}^{i}_{j_k} - \frac{t}{n} \cdot \mbf{u}^i \right\|\le 4d(d+1). 
\end{equation*}

Let $t = \lfloor\frac{n}{d+1}\rfloor$ and
let $\lambda \in [0,1]$ be defined by the equation
\[
\lambda \cdot \frac{\left\lfloor\frac{n}{d+1}\right\rfloor}{n} + (1-\lambda)\cdot\frac{\left\lfloor\frac{n}{d+1}\right\rfloor+1}{n} = \frac{1}{d+1}.
\]
For each $i \in \{1, \ldots, d+1\}$ and any $\floor{\sfrac{n}{(d+1)}}$ consecutive values (modulo $n$) in $\{1, \ldots, n\}$, say the values are $1,\ldots,\floor{\sfrac{n}{(d+1)}}$, we have
\begin{align}\label{eqConsecutiveSteinitz}
&\left\|\sum_{j=1}^{ \left\lfloor\tfrac{n}{d+1}\right\rfloor} \mbf{u}^{i}_{j} - \frac{1}{d+1}  \cdot\mbf{u}^i\right\|\notag\\[.25 cm]
=~& \left\|\lambda \left(\sum_{j=1}^{\left\lfloor\tfrac{n}{d+1}\right\rfloor} \mbf{u}^{i}_{j}  - \frac{\left\lfloor\tfrac{n}{d+1}\right\rfloor}{n} \cdot  \mbf{u}^i \right) +(1-\lambda) \left(\sum_{j=1}^{  \left\lfloor\tfrac{n}{d+1}\right\rfloor+1} \mbf{u}^{i}_{j}  - \frac{\left\lfloor\tfrac{n}{d+1}\right\rfloor+1}{n} \cdot \mbf{u}^{i} \right)\right. \cdots\notag\\[.25 cm]
& \hspace{3 in}\cdots\left.- (1-\lambda) \mbf{u}^{i}_{\left\lfloor\tfrac{n}{d+1}\right\rfloor+1}\right\| \notag\\[.25 cm]
\le ~& \lambda 4d(d+1) + (1-\lambda) 4d(d+1)+\left\|\mbf{u}^{i}_{\left\lfloor\tfrac{n}{d+1}\right\rfloor+1}\right\| \notag\\
\le ~& 4d(d+1)+1.
\end{align}

For each color $j \in \{1, \dotsc, n\}$, define the permutation $\tilde{\sigma}_j\in \mcf{S}^m$ as follows:
\[
\tilde{\sigma}_j(i) := 
\begin{cases}
\left(i+\left\lfloor \frac{j-1}{\floor{n/(d+1)}} \right\rfloor\right) \mod (d+1), & \text{if}~i \in \{1, \ldots,  d+1\}\\
i, & \text{if}~i \in\{d+2, \ldots, m\}.
\end{cases}
\]

\begin{figure}
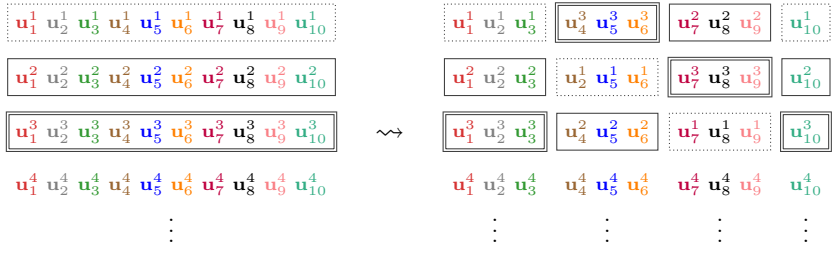

\centering
\begin{tabular}{c@{\hskip 0.5 cm}c@{\hskip 0.5 cm}c@{\hskip 0.1 cm}c@{\hskip 0.1 cm}c@{\hskip 0.1 cm}c}
\tikz[baseline = -3]{\node[draw = black!75, densely dotted, align = center]{\scriptsize${\color{red!80!black!80}\mbf{u}^{1}_1}\ {\color{black!50}\mbf{u}^{1}_{2}}\ {\color{green!50!black!80}\mbf{u}^{1}_{3}}\ {\color{orange!50!black!80}\mbf{u}^{1}_{4}}\ {\color{blue}\mbf{u}^{1}_{5}}\ {\color{orange}\mbf{u}^{1}_{6}}\ {\color{purple}\mbf{u}^{1}_{7}}\ \mbf{u}^{1}_{8}\ {\color{purple!20!red!50}\mbf{u}^{1}_{9}}\ {\color{blue!40!green!80}\mbf{u}^{1}_{10}}$};}
&
&
\tikz[baseline = -3]{\node[draw = black!75, densely dotted,  align = center]{\scriptsize${\color{red!80!black!80}\mbf{u}^{1}_1}\ {\color{black!50}\mbf{u}^{1}_{2}}\ {\color{green!50!black!80}\mbf{u}^{1}_{3}}$};}
&
\tikz[baseline = -3]{\node[draw = black!75, double,  align = center]{\scriptsize${\color{orange!50!black!80}\mbf{u}^{3}_{4}}\ {\color{blue}\mbf{u}^{3}_{5}}\ {\color{orange}\mbf{u}^{3}_{6}}$};}
&
\tikz[baseline = -3]{\node[draw = black!75,  align = center]{\scriptsize${\color{purple}\mbf{u}^{2}_{7}}\ \mbf{u}^{2}_8\ {\color{purple!20!red!50}\mbf{u}^{2}_{9}}$};}
&
\tikz[baseline = -3]{\node[draw = black!75, densely dotted,  align = center]{\scriptsize${\color{blue!40!green!80}\mbf{u}^{1}_{10}}$};}
\\[.25 cm]
\tikz[baseline = -3]{\node[draw = black!75,  align = center]{\scriptsize${\color{red!80!black!80}\mbf{u}^{2}_1}\ {\color{black!50}\mbf{u}^{2}_{2}}\ {\color{green!50!black!80}\mbf{u}^{2}_{3}}\ {\color{orange!50!black!80}\mbf{u}^{2}_{4}}\ {\color{blue}\mbf{u}^{2}_{5}}\ {\color{orange}\mbf{u}^{2}_{6}}\ {\color{purple}\mbf{u}^{2}_{7}}\ \mbf{u}^{2}_{8}\ {\color{purple!20!red!50}\mbf{u}^{2}_{9}}\ {\color{blue!40!green!80}\mbf{u}^{2}_{10}}$};}
&
&
\tikz[baseline = -3]{\node[draw = black!75,  align = center]{\scriptsize${\color{red!80!black!80}\mbf{u}^{2}_1}\ {\color{black!50}\mbf{u}^{2}_{2}}\ {\color{green!50!black!80}\mbf{u}^{2}_{3}}$};}
&
\tikz[baseline = -3]{\node[draw = black!75, densely dotted,  align = center]{\scriptsize${\color{orange!50!black!80}\mbf{u}^{1}_{2}}\ {\color{blue}\mbf{u}^{1}_{5}}\ {\color{orange}\mbf{u}^{1}_{6}}$};}
&
\tikz[baseline = -3]{\node[draw = black!75, double, align = center]{\scriptsize${\color{purple}\mbf{u}^{3}_{7}}\ \mbf{u}^{3}_8\ {\color{purple!20!red!50}\mbf{u}^{3}_{9}}$};}
&
\tikz[baseline = -3]{\node[draw = black!75,  align = center]{\scriptsize${\color{blue!40!green!80}\mbf{u}^{2}_{10}}$};}
\\[.25 cm]
\tikz[baseline = -3]{\node[draw = black!75, double,  align = center]{\scriptsize${\color{red!80!black!80}\mbf{u}^{3}_1}\ {\color{black!50}\mbf{u}^{3}_{2}}\ {\color{green!50!black!80}\mbf{u}^{3}_{3}}\ {\color{orange!50!black!80}\mbf{u}^{3}_{4}}\ {\color{blue}\mbf{u}^{3}_{5}}\ {\color{orange}\mbf{u}^{3}_{6}}\ {\color{purple}\mbf{u}^{3}_{7}}\ \mbf{u}^{3}_{8}\ {\color{purple!20!red!50}\mbf{u}^{3}_{9}}\ {\color{blue!40!green!80}\mbf{u}^{3}_{10}}$};}
&
\large$ \rightsquigarrow$
&
\tikz[baseline = -3]{\node[draw = black!75, double,  align = center]{\scriptsize${\color{red!80!black!80}\mbf{u}^{3}_1}\ {\color{black!50}\mbf{u}^{3}_{2}}\ {\color{green!50!black!80}\mbf{u}^{3}_{3}}$};}
&
\tikz[baseline = -3]{\node[draw = black!75,  align = center]{\scriptsize${\color{orange!50!black!80}\mbf{u}^{2}_{4}}\ {\color{blue}\mbf{u}^{2}_{5}}\ {\color{orange}\mbf{u}^{2}_{6}}$};}
&
\tikz[baseline = -3]{\node[draw = black!75, densely dotted,  align = center]{\scriptsize${\color{purple}\mbf{u}^{1}_{7}}\ \mbf{u}^{1}_8\ {\color{purple!20!red!50}\mbf{u}^{1}_{9}}$};}
&
\tikz[baseline = -3]{\node[draw = black!75, double,  align = center]{\scriptsize${\color{blue!40!green!80}\mbf{u}^{3}_{10}}$};}
\\[.25 cm]
\tikz[baseline = -3]{\node[draw = none,  align = center]{\scriptsize${\color{red!80!black!80}\mbf{u}^{4}_1}\ {\color{black!50}\mbf{u}^{4}_{2}}\ {\color{green!50!black!80}\mbf{u}^{4}_{3}}\ {\color{orange!50!black!80}\mbf{u}^{4}_{4}}\ {\color{blue}\mbf{u}^{4}_{5}}\ {\color{orange}\mbf{u}^{4}_{6}}\ {\color{purple}\mbf{u}^{4}_{7}}\ \mbf{u}^{4}_{8}\ {\color{purple!20!red!50}\mbf{u}^{4}_{9}}\ {\color{blue!40!green!80}\mbf{u}^{4}_{10}}$};}
&
&
\tikz[baseline = -3]{\node[draw =none,  align = center]{\scriptsize${\color{red!80!black!80}\mbf{u}^{4}_1}\ {\color{black!50}\mbf{u}^{4}_{2}}\ {\color{green!50!black!80}\mbf{u}^{4}_{3}}$};}
&
\tikz[baseline = -3]{\node[draw = none,  align = center]{\scriptsize${\color{orange!50!black!80}\mbf{u}^{4}_{4}}\ {\color{blue}\mbf{u}^{4}_{5}}\ {\color{orange}\mbf{u}^{4}_{6}}$};}
&
\tikz[baseline = -3]{\node[draw = none,  align = center]{\scriptsize${\color{purple}\mbf{u}^{4}_{7}}\ \mbf{u}^{4}_8\ {\color{purple!20!red!50}\mbf{u}^{4}_{9}}$};}
&
\tikz[baseline = -3]{\node[draw = none,  align = center]{\scriptsize${\color{blue!40!green!80}\mbf{u}^{4}_{10}}$};}
\\
$\vdots$
&&
$\vdots$&$\vdots$&$\vdots$&$\vdots$
\end{tabular}
\caption{Suppose $d+1=3$, $n=10$, and $m$ is arbitrary.
On the lefthand side of the $\rightsquigarrow$ symbol, the vectors $(\mbf{u}^i_j)_{i,j}$ are arranged in an $m\times n$ matrix.
The vectors in color $j$ are arranged in the $j$th column of the matrix, and the $i$th row is the sequence $(\mbf{u}^{i}_j)_{j=1}^n$.
The permutation $\tilde{\sigma}_j$ rearranges the $j$th color.
Rows $1, 2$, and $3$ are surrounded with different boxes (dotted line, solid line, and double line) to illustrate how the permutations $\tilde{\sigma}_1, \dotsc, \tilde{\sigma}_n$ affect each row.
On the righthand side of the $\rightsquigarrow$ symbol, we again have an $m\times n$ matrix whose entries are the vectors $\mbf{u}^i_j$.
The colors still correspond to the columns, but now the $i$th row is the sequence $(\mbf{u}^{\tilde{\sigma}_j(i)}_j)_{j=1}^n$.
For $i \in \{1,\dotsc, d+1\}$, the $i$th row consists of $\lfloor \sfrac{n}{(d+1)}\rfloor$ consecutive colors from each sequence $(\mbf{u}^{i'}_j)_{j=1}^n$ for each $i' \in \{1,\dotsc, d+1\}$; this can be seen as each row is now composed of parts from the different boxes from the lefthand side matrix.
According to~\eqref{eqConsecutiveSteinitz}, each set of $\lfloor \sfrac{n}{(d+1)}\rfloor$ consecutive colors in a box, say the superscript corresponding to the box is $i$, differs from $\sfrac{1}{(d+1)}\cdot \mbf{u}^i$ in the $\|\cdot\|$ norm by an additive factor of $4d(d+1)+1$.
The $i$th row ends with $n- (d+1)\lfloor \sfrac{n}{(d+1)} \rfloor$ consecutive colors from $(\mbf{u}^{i}_j)_{j=1}^n$.
Rows $d+2, \ldots, m$ corresponding to the sequences $(\mbf{u}^{d+2}_j)_{j=1}^n, \ldots, (\mbf{u}^{m}_j)_{j=1}^n$ are not changed by $\tilde{\sigma}_1, \ldots, \tilde{\sigma}_n$.
}\label{figRearrange}
\end{figure}

Figure~\ref{figRearrange} illustrates $\tilde{\sigma}_1, \ldots, \tilde{\sigma}_n$ when $d+1 = 3$ and $n = 10$.
For each $i \in\{ 1, \ldots, d+1\}$, the sequence $(\mbf{u}^{\tilde{\sigma}_j(i)}_j)_{j=1}^n$ consists of exactly $\lfloor\sfrac{n}{(d+1)}\rfloor$ consecutive colors from each sequence $(\mbf{u}^{i'}_j)_{j=1}^n$ for each $i' \in \{ 1, \ldots, d+1\} $ and an additional $n - (d+1) \cdot \lfloor\sfrac{n}{(d+1)}\rfloor < d+1$ consecutive colors from the sequence $(\mbf{u}^{i}_j)_{j=1}^n$.
For each $i \in\{ 1, \ldots, d+1\}$, we use~\eqref{eqCentroid1} and~\eqref{eqConsecutiveSteinitz} to see that
\begin{align*}
&\left\|\sum_{j=1}^n \mbf{u}^{\tilde{\sigma}_j(i)}_j\right\|\\
=& \left\|\sum_{j=1}^{\lfloor\sfrac{n}{(d+1)}\rfloor} \mbf{u}^i_j+\sum_{j=1}^{\lfloor\sfrac{n}{(d+1)}\rfloor}\mbf{u}^{i+1}_j+ \cdots + \sum_{j=1}^{\lfloor\sfrac{n}{(d+1)}\rfloor}\mbf{u}^{(i+d+1)\ {\rm mod}\ n}_j+\sum_{j=(d+1)\lfloor \frac{n}{d+1}\rfloor}^n \mbf{u}^{i}_j \right\|\\
\le &   (d+1)(4d(d+1)+1) + \left\|\sum_{i=1}^{d+1} \frac{1}{d+1}\cdot  \mbf{u}^{i}\right\| + \left\|\sum_{j=(d+1)\lfloor \frac{n}{d+1}\rfloor}^n \mbf{u}^{i}_j \right\|\\
= &   (d+1)(4d(d+1)+1) + \left\|\mbf{c}\right\| + \left\|\sum_{j=(d+1)\lfloor \frac{n}{d+1}\rfloor}^n \mbf{u}^{i}_j \right\|\\
\le  &   (d+1)(4d(d+1)+1) + \left\|\mbf{c}\right\| +(d+1).
\end{align*}
Recalling the bound on $\|\mbf{u}^1\|$ in~\eqref{eqTOObig} and the bound on $\|\mbf{c}\|$ in~\eqref{eqBoundC}, we have:
\[
\left\|\sum_{j=1}^n \mbf{u}^{\tilde{\sigma}_j(i)}_j\right\|
\le  (d+1)(4d(d+1)+2)  +\frac{d}{d+1}  \left\|\mbf{u}^1\right\|
<  \left\|\mbf{u}^1\right\|.
\]
For each $i \in \{d+2, \ldots, m\}$ we have $\tilde{\sigma}_j(i) = i$ by definition, so $\|\sum_{j=1}^n \mbf{u}^{\tilde{\sigma}_j(i)}_j\| = \|\mbf{u}^i\| \le \|\mbf{u}^1\|$.
Hence, either $\tilde{\sigma}_1, \ldots, \tilde{\sigma}_n$ yield a smaller objective value in~\eqref{eqSteinitzMin} than ${\sigma}_1, \ldots, {\sigma}_n$, or they have the same objective value but fewer indices $i$ that attain the maximum value. 
Both situations are contradictions. 
\hfill ~\qed
\end{proof}

Our proof of Theorem~\ref{thVarSteinitz} is inspired by Theorem 1 in Ambrus et al.~\cite{ABG2016}, which cleverly utilizes linear optimization and rounding techniques.  
Grinberg and Sevastyanov also use linear optimization techniques to prove the classical Steinitz Lemma (Theorem~\ref{thmSteinitz}), but it is not clear how to leverage their approach to the colorful setting (Theorem~\ref{thmVarSteinitz}).
Our proof starts with an auxiliary lemma.

\begin{lemma}\label{lem:rounding}
Let $\mcf{Q} :=\{\mbfs{\alpha}  \in [0,1]^m :\ \|\mbfs{\alpha} \|_1=k\}$, where $k \in \{0, \ldots, m\}$.
For every $\mbf{y} \in \mcf{Q}$, there exists $\mbf{z} \in \mcf{Q}\cap\{0,1\}^m$ such that $\|\mbf{y}-\mbf{z}\|_1\le \sfrac{m}{2}$.
\end{lemma}

\begin{proof}
Let $\mbf{y} = (y_i) \in \mcf{Q}$.
Without loss of generality, suppose $y_1\ge \cdots\ge y_m$. 
Define $\mbf{z}  = (z_i) \in \mcf{Q} \cap \{0,1\}^m$ so that $z_i = 1$ for $i \in \{1, \dotsc, k\}$ and $z_i = 0$ else.
Note that $y_1, \dotsc, y_k\ge \sfrac{k}{m}$, wich gives us
\[
\|\mbf{y}- \mbf{z}\|_1 = \sum_{i=1}^k (1-y_i) + \sum_{i=k+1}^m y_i  = 2k- 2\sum_{i=1}^ky_i \le 2k - \frac{2k^2}{m} \le \frac{m}{2}.
\]
\hfill ~\qed
\end{proof}

\begin{proof}[of Theorem~\ref{thVarSteinitz}.]
Set $\mcf{Q}:=\{\mbfs{\alpha}\in[0,1]^m :\ \|\mbfs{\alpha}\|_1=k\}$ and 
\[
\mcf{P} :=\left\{ (\mbfs{\alpha}_1, \ldots, \mbfs{\alpha}_n) \in \underbrace{\mcf{Q}\times \cdots \times \mcf{Q}}_{n~\text{times}} \;:\; \sum_{i=1}^m\sum_{j=1}^n \alpha^i_j \mbf{u}^{i}_j = \mbf{0} \right\}.
\]
For every $\gamma \in \mbb{R}$, the sequence $(\gamma \mbf{u}^{i}_j)_{i,j}$ is a zero-sum sequence because $(\mbf{u}^{i}_j)_{i,j}$ is a zero-sum sequence. 
Setting $\gamma = \sfrac{k}{m}$ shows that $\mcf{P} \neq \emptyset$.
Let $\overline{\mbfs{\alpha}}=(\overline{\mbfs{\alpha}}_1, \ldots, \overline{\mbfs{\alpha}}_n)$ be a vertex of $\mcf{P}$.
Without loss of generality, we assume $\overline{\mbfs{\alpha}}_1, \ldots, \overline{\mbfs{\alpha}}_{\ell} \not \in \{0,1\}^m$ and $\overline{\mbfs{\alpha}}_{\ell+1}, \ldots, \overline{\mbfs{\alpha}}_{n} \in \{0,1\}^m$.

We claim $\overline{\mbfs{\alpha}}$ has at most $2d$ fractional entries.
For each $j \in \{ 1, \ldots, \ell\}$, we assume the factional components of $\overline{\mbfs{\alpha}}_j = (\overline{{\alpha}}^i_j)_{i=1}^m $ are indexed by $i=1,\ldots,k_j$;
we have $k_j \ge 2$ because $\overline{\mbfs{\alpha}}_j \in \mcf{Q}$.
Our claim is equivalent to $\sum_{j=1}^\ell k_j \le 2d$.
%
%
For each $j\in \{1,\ldots,\ell\}$ and $i\in \{1,\ldots,k_j-1\}$, define a perturbation 
\[
\mbf{p}^{i}_j\ = (p^{i}_{j,k})_{k=1}^{mn} 
\]
component-wise by
\[
p^{i}_{j,k}=
\left\{
\begin{array}{r@{\hskip .25 cm}l}
1 & \text{if} \quad k = (j-1)m+i \\[.1 cm]
 -1 & \text{if} \quad k = (j-1)m+i +1 \\[.1 cm]
  0 & \text{otherwise}. 
\end{array}\right.
\]
The $\sum_{j=1}^\ell (k_j-1)$ vectors in the sequence $(\mbf{p}^{i}_j)_{i,j}$ are linearly independent. 

We have $2\ell \le \sum_{j=1}^{\ell} k_j$ because $k_1, \dotsc, k_{\ell}\ge2$.
We claim that $\sum_{j=1}^\ell k_j\le d+\ell$, which will imply $\ell \le d$ and   $\sum_{j=1}^\ell k_j\le 2d$.
Assume to the contrary that $\sum_{j=1}^\ell k_j > d+\ell$, or equivalently that $\sum_{j=1}^\ell (k_j-1)>d$.
Recall that each $\mbf{u}^i_j$ is in $\mbb{R}^d$.
Thus, the $\sum_{j=1}^\ell (k_j-1)$ vectors $\sum_{e=1}^m\sum_{f=1}^n p^{i}_{j,(e-1)m+f} \mbf{u}^{e}_f$, which are indexed by $j\in\{1,\ldots,\ell\}$ and $i\in\{1,\ldots,k_j-1\}$, are linearly dependent.
There exists a sequence of numbers $\lambda^{i}_j$, which again is indexed by $j\in\{1,\ldots,\ell\}$ and $i\in\{1,\ldots,k_j-1\}$,
that are not all zero and satisfy 
\[
\sum_{e=1}^m\sum_{f=1}^n \left(\sum_{j=1}^\ell\sum_{i=1}^{k_j-1}\lambda^{i}_j \cdot p^{i}_{j,(e-1)m+f} \right)\mbf{u}^{e}_f= \mbf{0}.
\]
Set $\mbf{p}:=\sum_{j=1}^\ell\sum_{i=1}^{k_j-1} \lambda^{i}_j \mbf{p}^{i}_j$.
The vector $\mbf{p}$ is only supported on those components where $\overline{\mbfs{\alpha}}$ is non-integer.
From the previous equation we can choose $\epsilon > 0 $ sufficiently small so that $\overline{\mbfs{\alpha}} \pm \epsilon \mbf{p} \in \mcf{P}$, contradicting that $\overline{\mbfs{\alpha}}$ is a vertex.
Therefore, $\sum_{j=1}^\ell k_j \le 2d$.
%

For each $j \in\{ 1, \ldots, \ell\}$, we can apply Lemma~\ref{lem:rounding} with $m=k_j$ and $k=\sum_{i=1}^{k_j} \overline{\alpha}^i_j$ to round the fractional components of $\overline{\mbfs{\alpha}}_j$ to 0 or 1, calling the resulting rounded vector $\lfloor\overline{\mbfs{\alpha}}_j\rceil$, such that $\|\lfloor\overline{\mbfs{\alpha}}_j\rceil\|_1=k$ and $\| \lfloor\overline{\mbfs{\alpha}}_j\rceil -\overline{\mbfs{\alpha}}_j \|_1\le \sfrac{k_j}{2}$.
Set
\(
I_j:= \supp (\lfloor\overline{\mbfs{\alpha}}_j\rceil)
\)
for each $j \in \{1, \ldots, \ell\}$ and $I_j := \supp(\overline{\mbfs{\alpha}}_j)$ for each $j \in \{ \ell+1, \ldots, n\}$.
We see that
\[
\begin{array}{rcl@{\hskip 1 cm}l}
\displaystyle\left\|\sum_{j=1}^n \sum_{i \in I_j}  \mbf{u}^i_j \right\| &=& \displaystyle\left\|\sum_{j=1}^{\ell} \sum_{i \in I_j} \left(\lfloor\overline{\alpha}^i_j\rceil-\overline{\alpha}^i_j\right) \mbf{u}^i_j +\sum_{j=1}^{\ell} \sum_{i \in I_j} \overline{\alpha}^i_j \mbf{u}^i_j + \sum_{j=\ell+1}^{n} \sum_{i \in I_j} \overline{\alpha}^i_j \mbf{u}^i_j \right\|\\[.75 cm]
&= &\displaystyle\left\|\sum_{j=1}^{\ell} \sum_{i \in I_j} \left(\lfloor\overline{\alpha}^i_j\rceil-\overline{\alpha}^i_j\right) \mbf{u}^i_j  \right\|\\[.75 cm]
& \le &\displaystyle \sum_{j=1}^{\ell} \frac{k_j}{2} \le \frac{2d}{2} = d. &
\end{array}
\]
\hfill ~\qed
\end{proof}

In the proof of Theorem~\ref{thm4BlockVector} we use the classic Steinitz Lemma in linear subspaces.
The proof of Lemma~\ref{lemmaSteinitzLowRank} follows directly from Theorem~\ref{thmSteinitz}.
We provide a proof for completeness.

\begin{lemma}\label{lemmaSteinitzLowRank}
Let $\|\cdot\|: \mbb{R}^d \to \mbb{R}$ be a norm with unit ball $\mcf{U}$.
Let $\mcf{V} \subseteq \mbb{R}^d$ be a linear subspace.
Let $(\mbf{u}^i)_{i=1}^m$ be a zero-sum sequence in $\mcf{U} \cap \mcf{V}$. 
There exists a permutation $\pi \in \mcf{S}^m$ such that 
\[
\left\|\sum_{i=1}^k \mbf{u}^{\pi(i)}  \right\| \le \dim(\mcf{V})
\] 
for each $k \in \{ 1, \ldots, m\}$.
\end{lemma}
\begin{proof}
Without loss of generality we may assume $\mcf{V}= \Span\{\mbf{u}^i:\ i \in \{ 1, \ldots, m\}\}$.
Let $\phi: \mcf{V} \to \mbb{R}^{\dim(\mcf{V}) }$ be a linear bijection.
Define a norm $\|\cdot\|_{*}$ on $\mbb{R}^{\dim(\mcf{V}) }$ by 
\[
\|\phi(\mbf{w})\|_{*} := \|\mbf{w}\| \qquad \forall ~\mbf{w} \in \mcf{V}.
\]

The sequence $(\phi(\mbf{u}^i))_{i=1}^m$ in the unit ball $\phi(\mcf{U})$ is a zero-sum sequence because $\phi$ is linear. 
By the Steinitz Lemma (Theorem~\ref{thmSteinitz}) applied to $(\phi(\mbf{u}^i))_{i=1}^m$ there exists a permutation $\pi \in \mcf{S}^m$ such that 
\[
\left\|\sum_{i=1}^k \mbf{u}^{\pi(i)}\right\|
=
\left\|\phi\left(\sum_{i=1}^k \mbf{u}^{\pi(i)}\right)\right\|_{*}
=
\left\|\sum_{i=1}^k \phi\left(\mbf{u}^{\pi(i)}\right)\right\|_{*} \le \dim(\mcf{V})  
\]
for each $ k \in \{1, \ldots, m\}$.
\hfill ~\qed
\end{proof}


\section{A proof of Theorem~\ref{thm4BlockVector}.}\label{secMainProof}

By studying the 4-block matrix in the generality presented in~\eqref{eq3BlockMatrix}, we may assume without loss of generality that
\begin{equation}\label{eqAssumeA0}
\text{$\mbf{A}^0 = \mbf{0}$.}
\end{equation}
Indeed, $(\mbf{x}, \mbf{y})$ is in the kernel of the 4-block matrix~\eqref{eq3BlockMatrix} if and only if $(\mbf{x}, \mbf{x}, \mbf{y})$ is in the kernel of the {\it 3-block matrix}
\[\left[
\begin{array}{c|l@{\hskip .15 cm}c@{\hskip .15 cm}c@{\hskip .15 cm}c}
 \mbf{0}^{\phantom{1}} & \phantom{-}\mbf{A}^0 &\mbf{C}^1 & \cdots & \mbf{C}^n\\
 \hline&\\[-.35 cm]
 \mathbf{I}^{\phantom{1}}& - \mathbf{I}\\
 \mbf{B}^1 && \mbf{A}^1 \\
 \vdots & && \ddots\\
 \mbf{B}^n & &&&\mbf{A}^n
\end{array}
\right].
\]
For the remainder of the proof we assume~\eqref{eqAssumeA0}.
This assumption decouples the linking variables $\mbf{x}$ and the linking constraints $ [\mbf{C}^1 ~ \cdots ~ \mbf{C}^n]$.

In the proof we write $\mbf{w} \in  \mbb{R}^{nt}$ as $\mbf{w} = (\mbf{w}^1, \ldots, \mbf{w}^n)$, where $\mbf{w}^1, \ldots,\mbf{w}^n \in \mbb{R}^t$.
We also use the notation
\[
\mbf{H}=\left[\begin{array}{c|r}\mbf{0}&\mbf{C}\\\hline \mbf{B}&\mbf{A}\end{array}\right] = 
\left[
\begin{array}{c|c@{\hskip .15 cm}c@{\hskip .15 cm}c@{\hskip .15 cm}c}
\mbf{0} & \mbf{C}^1 & \cdots & \mbf{C}^n\\
 \hline &\\[-.35 cm]
 \mbf{B}^1 & \mbf{A}^1 \\
 \vdots & & \ddots\\
 \mbf{B}^n & &&\mbf{A}^n
\end{array}
\right].
\]

Let $(\widehat{\mbf{x}}, \widehat{\mbf{y}})\in (\mbb{R}_+^{t_0} \times \mbb{R}_+^{nt})\cap  \ker \mbf{H}$. 
Suppose
\begin{equation}\label{eqXi}
\|(\widehat{\mbf{x}},\widehat{\mbf{y}})\|_{\infty} > \xi,
\end{equation}
where $\xi \in \fpt(n^{\min\{t_0+2,s_0\}})$ is to be defined in~\eqref{eqLargeConstant}.
We will prove that there exists a nonzero vector $({\mbf{x}},{\mbf{y}}) \in (\mbb{Z}_+^{t_0} \times \mbb{Z}_+^{nt})\cap  \ker \mbf{H}$ such that $({\mbf{x}},{\mbf{y}})  \le (\widehat{\mbf{x}}, \widehat{\mbf{y}})$.

For many results of the same nature as Theorem~\ref{thm4BlockVector}, the core component in the proof is based on a decomposition of the target vector $(\widehat{\mbf{x}}, \widehat{\mbf{y}})$ into a sequence of `smaller' vectors that can be rearranged so that a subsequence sums to the desired $(\mbf{x},\mbf{y})$.
Using this proof technique, Klein~\cite{K2020} decomposes integer-valued target vectors in the 2-stage setting to obtain a tight bound on the size of Graver basis elements.
In the setting of $3$-block matrices, Klein's decomposition can be used to obtain a bound of $\xi \in \fpt(n^{s_0})$; see Chen et al.~\cite[Theorem 1]{CKXS2019}.
In a related context, let us mention Chen et al.~\cite[Theorem 5]{CKXS2019}, who establish a decomposition technique of integer-valued target vectors in the 3-block setting.
Building on Klein's work and assuming the block matrices are all equal, i.e. $\mbf{A}^1 = \cdots =\mbf{A}^n$, $\mbf{B}^1 = \cdots =\mbf{B}^n$, and $\mbf{C}^1 = \cdots = \mbf{C}^n$, Chen et al. redistribute the entries of the target vector {\it across} the blocks to decompose the target into vectors with small components everywhere except in one of the $n$ blocks; one of the key ingredients is a `merging lemma'~\cite[Lemma 5]{CKXS2019}.
Using this clever redistribution, they are able to rearrange the small vectors to obtain $\xi \in \fpt(n^{t^2})$. 
Our target is to shed light on the joint complexity of the number $s_0$ of linking constraints and (in particular) the number $t_0$ of linking variables.
We improve on $\xi \in \fpt(n^{s_0})$ by providing a refined analysis of the decomposition presented in Klein's work: in order to keep control on the contributions in the linking constraints we rearrange elements {\it within} blocks with the help of the colorful Steinitz Lemma.
%
Moreover, our analysis considers arbitrary blocks $\mbf{A}^i$, $\mbf{B}^i$, and $\mbf{C}^i$, as well as target vectors that are not necessarily integer-valued. 
The ability to apply our analysis to integer and fractional target 
vectors allows us to simultaneously obtain bounds on the proximity 
between optimal LP- and IP-solutions as well as on the size of Graver 
basis elements.


%
%

To present our proof, we first state the decomposition together with all of the technical assumptions. 
We then prove Theorem~\ref{thm4BlockVector}, and the details of the technical assumptions are left to later subsections. 

To begin our decomposition, we write $(\widehat{\mbf{x}}, \widehat{\mbf{y}})= (\mbf{0}, \widehat{\mbf{u}}) + (\widehat{\mbf{x}}, \widehat{\mbf{v}})$ where $\widehat{\mbf{u}} \in  \mbb{R}^{nt}_+\cap \ker \mbf{A}$ and $\widehat{\mbf{v}} \in \mbb{R}^{nt}_+$.
%
We assume that $\widehat{\mbf{u}}$ is maximal:
\begin{equation}\label{assumeNoMoreu}
\text{there does not exist $\widehat{\mbf{r}} \in (\mbb{R}_+^{nt} \setminus \{\mbf{0}\}) \cap \ker \mbf{A}$ with $\widehat{\mbf{r}} \le \widehat{\mbf{v}} $.}
\end{equation}
In particular, this implies that 
\begin{equation}\label{equ:boundVbyX}
\|\widehat{\mbf{v}}\|_\infty \le \omega_1 \|\widehat{\mbf{x}}\|_\infty,
\end{equation}
where $\omega_1\in\fpt(1)$, see Part b) in Lemma~\ref{lemMinkSum} for details.
We further decompose $\widehat{\mbf{u}},\;\widehat{\mbf{x}}$ and $\widehat{\mbf{v}}$ as follows:
\begin{align*}
(\widehat{\mbf{x}}, \widehat{\mbf{y}})
&= (\mbf{0}, \widehat{\mbf{u}}) \quad\quad\quad + \left(\vphantom{\int}\smash{\underbrace{\sum_{\ell=1}^{t_0} \lambda_\ell \mbf{h}^\ell}_{=\widehat{\mbf{x}}}},\ \vphantom{\int}\smash{\underbrace{\sum_{\ell = 1}^{t_0} \mbf{v}_{j}}_{=\widehat{\mbf{v}}}}\right)\\[0.6cm]
 &= \left(\mbf{0}, \vphantom{\int}\smash{\underbrace{\sum_{j=0}^{{\alpha_0}} \widehat{\mbf{u}}_j}_{=\widehat{\mbf{u}}}}\right) + \left(\sum_{\ell=1}^{t_0} \lambda_\ell \mbf{h}^\ell,\ \sum_{\ell = 1}^{t_0}\vphantom{\int}\smash{\underbrace{\sum_{j=0}^{\alpha_{\ell}} \mbf{v}_{\ell,j}}_{={ \mbf{v}_{j}}}}\right)\\
\end{align*}
where we require the properties listed below:
%
%


\begin{minipage}{.9\textwidth}
\schema[]{
\begin{enumerate}[label = \roman*)]
\item\label{technicalKeyProperties1}
${\alpha_0}\in\mbb{Z}_+$ and
${\alpha_0} \ge \frac{\|\widehat{\mbf{u}}\|_\infty}{t(2s\Delta+1)^s}-1.$
\smallskip
\item\label{technicalKeyProperties2}
$\widehat{\mbf{u}}_j \in \mbb{Z}^{nt}_+ \cap \ker \mbf{A}$ and $\|\widehat{\mbf{u}}_j \|_{1}\in\fpt(1)$ for $j \in \{1, \ldots, {\alpha_0}\}$.
\smallskip
\item\label{technicalKeyProperties3}
$\widehat{\mbf{u}}_{0} \in \mbb{R}^{nt}_+\cap \ker \mbf{A}$, $\|\widehat{\mbf{u}}_{0} \|_{1}\in\fpt(n)$ and $\|\widehat{\mbf{u}}_{0} \|_{\infty} \in \fpt(1)$.
\smallskip
\item\label{technicalKeyProperties4}
$\| \sum_{j=1}^k \mbf{C}\widehat{\mbf{u}}_j  - \frac{k}{{\alpha_0}} \sum_{j=1}^{\alpha_0}\mbf{C}\widehat{\mbf{u}}_j \|_{\infty} \in \fpt(1)$
for each $ k \in \{1, \ldots, {\alpha_0}\}$.
%
%
\end{enumerate}}
{\begin{turn}{-90}\text{Lemma~\ref{lemDecomposeU}}\end{turn}}
\end{minipage}

\begin{minipage}{.9\textwidth}
\schema[]{
For all $\ell \in \{1, \dotsc, t_0\}$,
\begin{enumerate}[label = \roman*)]
\setcounter{enumi}{4}
\item\label{technicalKeyProperties5}
$ \lambda_{\ell} \in \mbb{R}_+$.
\smallskip
\item\label{technicalKeyProperties6}
$\mbf{h}^\ell \in \mbb{Z}^{t_0}_+$ and $\|\mbf{h}^{\ell}\|_{\infty} \le \omega_2 \in \fpt(1)$.
\end{enumerate}

}
{\begin{turn}{-90}\text{Lemma~\ref{lemDecomposeX}}\end{turn}}
\end{minipage}

\begin{minipage}{.9\textwidth}
\schema[]{
%
%
\begin{enumerate}[label = \roman*)]
\setcounter{enumi}{6}
\item\label{technicalKeyProperties65}
$\sum_{\ell=1}^{t_0} \alpha_{\ell} \ge \frac{\|\widehat{\mbf{x}}\|_\infty}{{\omega_2}}-t_0(t-s+2) .
$
\smallskip
\end{enumerate}
For all $\ell \in \{1, \ldots, t_0\}$,
\begin{enumerate}[label = \roman*)]
\setcounter{enumi}{7}
\item\label{technicalKeyProperties7}
 $\mbf{v}_{\ell, j} \in \mbb{Z}^{nt}_+$,  $(\mbf{h}^{\ell}, \mbf{v}_{\ell, j}) \in \ker \begin{bmatrix}\mbf{B} & \mbf{A}\end{bmatrix}$ and\\ $\|\mbf{v}^{i}_{\ell, j}\|_1 \le  \fpt(1)$ for each $i \in \{1,\ldots, n\}$ and $j \in \{1, \ldots, \alpha_{\ell}\}$.

\smallskip
\item\label{technicalKeyProperties8}
$ \mbf{v}_{\ell,0} \in \mbb{R}^{nt}_+$ satisfies $\mbf{A} \mbf{v}_{\ell,0} \in \mbb{Z}^{ns}$ and\\ $\|\mbf{v}^{i}_{\ell,0} \|_{1} \in\fpt(1)$ for each $i \in \{1,\ldots, n\}$.
\smallskip
\item\label{technicalKeyProperties9} 
If $\alpha_{\ell} \ge 1$, then $\|\sum_{j=1}^k\mbf{C}\mbf{v}_{\ell,j} - \frac{k}{\alpha_{\ell}}\sum_{j=1}^{\alpha_{\ell}} \mbf{C}\mbf{v}_{\ell,j} \|_{\infty} \in\fpt(1)$  for each $k \in \{1, \ldots, \alpha_{\ell}\}$.
\end{enumerate}}
{\begin{turn}{-90}\text{Lemma~\ref{lemDecomposeV}}\end{turn}}
\end{minipage}

\smallskip

Note that one of our main contributions lies within \ref{technicalKeyProperties9}, where we use our colorful Steinitz Lemma.




%
%

\begin{proof}[of Theorem~\ref{thm4BlockVector}]
Using the above decomposition, we can write $\widehat{\mbf{y}} \in \ker \mbf{C}$ as
\begin{equation}\label{eqZeroSum}
\mbf{0} = \mbf{C}\widehat{\mbf{y}} = \sum_{\ell=1}^{t_0} \underbrace{\mbf{C}\left(\sum_{j=1}^{\alpha_{\ell}} \mbf{v}_{\ell,j}\right)}_{=:\mbf{p}^{\ell}}
+ \underbrace{\vphantom{\mbf{C}\left(\sum_{j=1}^{\alpha_{\ell}} \mbf{v}_{\ell,j}\right)}\sum_{\ell=1}^{t_0} \mbf{C}\mbf{v}_{\ell,0} 
+ \mbf{C} \widehat{\mbf{u}}_{0}}_{=: \mbf{r}}
+\underbrace{\displaystyle\vphantom{\mbf{C}\left(\sum_{j=1}^{\alpha_{\ell}} \mbf{v}_{\ell,j}\right)}\sum_{j=1}^{{\alpha_0}} \mbf{C} \widehat{\mbf{u}}_{j}}_{=:\mbf{q}},
\end{equation}
where $\mbf{r} ,\mbf{q},\mbf{p}^{1},\ldots, \mbf{p}^{t_0} \in \mbb{R}^{s_0}$.
Set
\[
\begin{array}{rcl}
\displaystyle\mcf{V} &:=&\displaystyle\Span \{\mbf{r}, \mbf{q}, \mbf{p}^{1}, \ldots, \mbf{p}^{t_0}\}\\[.15 cm]
\displaystyle \psi &:=&\displaystyle 1+{\alpha_0}+\sum_{\ell=1}^{t_0} \alpha_\ell .
\end{array}
\]
We have $\dim(\mcf{V}) \le \min\{s_0,t_0+2\}$ because $\mcf{V}$ is spanned by $t_0+2$ vectors in $\mbb{R}^{s_0}$. 

Our goal is to rearrange the integer vectors $\mbf{C}\mbf{v}_{\ell, j}$, $\mbf{C}\mbf{v}_{\ell, 0}$, $\mbf{C}\widehat{\mbf{u}}_0$ and $\mbf{C}\widehat{\mbf{u}}_j$ in~\eqref{eqZeroSum} so that they lie in a small region.
In order to achieve this, we rearrange the following sequence of fractional vectors:
\begin{equation}\label{eqDicedPieces}
 \underbrace{\frac{1}{\alpha_1} \mbf{p}^{1}, \ldots, \frac{1}{\alpha_1} \mbf{p}^{1}}_{\text{$\alpha_1$ copies}},
 ~ \ldots~ ,
 \underbrace{\frac{1}{\alpha_{t_0}} \mbf{p}^{t_0}, \ldots, \frac{1}{\alpha_{t_0}} \mbf{p}^{t_0}}_{\text{$\alpha_{t_0}$ copies}}, ~
  \underbrace{\frac{1}{{\alpha_0}} \mbf{q}, \ldots, \frac{1}{{\alpha_0}} \mbf{q}}_{\text{${\alpha_0}$ copies}},~
  \mbf{r}.
\end{equation}
These $\psi$ vectors form a zero-sum sequence in $\mcf{V}$.
Define 
\[
\omega_3 := \max \left\{ \frac{1}{\alpha_{1}}\left\| \mbf{p}^{1} \right\|_{\infty} ,\ldots,\frac{1}{\alpha_{t_0}}\left\| \mbf{p}^{t_0} \right\|_{\infty} ,\frac{1}{{\alpha_0}}\left\| \mbf{q} \right\|_{\infty} , \left\| \mbf{r} \right\|_{\infty}\right\}.
\]
Note $\omega_3\in\fpt(n)$ by Properties~\ref{technicalKeyProperties2}, \ref{technicalKeyProperties3}, \ref{technicalKeyProperties7} and~\ref{technicalKeyProperties8}.
Let $(\mbf{s}^j)_{j=1}^{\psi}$ denote the sequence in~\eqref{eqDicedPieces} and distinguish $\mbf{s}^{\psi} := \mbf{r}$.
By applying Lemma~\ref{lemmaSteinitzLowRank} to this sequence, we can conclude that there exists a permutation $\pi \in \mcf{S}^{\psi}$ such that
$\| \sum_{j=1}^k \mbf{s}^{\pi(j)} \|_{\infty} \le \omega_3 \dim(\mcf{V})$
for each $k \in \{1, \ldots, \psi\}$.
We may insist that $\pi(\psi)=\psi$ by introducing an additional error of $\omega_3$ for each partial sum.
More precisely, we have $\pi(\psi)=\psi$ and
\begin{equation}\label{eqRearrangeFractions}
\left\| \sum_{j=1}^k \mbf{s}^{\pi(j)} \right\|_{\infty} \le \omega_3 (\dim(\mcf{V})+1)\le\omega_3 \min\{s_0+1,t_0+3\}
\end{equation}
for each $k \in \{1, \ldots, \psi\}$.
%

For $k \in \{1, \ldots, \psi-1\}$ and $\ell \in \{1, \ldots, t_0\}$, we need to keep track of how many vectors in $(\mbf{s}^{\pi(j)})_{j=1}^k$ are of the type $\sfrac{1}{\alpha_\ell}\cdot \mbf{p}^\ell$ and how many are of the type $\sfrac{1}{{\alpha_0}} \cdot \mbf{q}$. 
Define
\[
\begin{array}{rcl}
\phi_{k,\ell} & := & \displaystyle\left| \left\{ j \in \{ 1, \ldots, k\} : \mbf{s}^{\pi(j)} = \frac{1}{\alpha_\ell}  \mbf{p}^\ell\right\}\right| \quad \forall\ \ell ~\in \{1, \ldots, t_0\}\\[.5 cm]
\mu_{k} & := &  \displaystyle\left| \left\{ j \in \{ 1, \ldots, k\}: \mbf{s}^{\pi(j)} = \frac{1}{{\alpha_0}}  \mbf{q}^{\phantom{\ell}}\right\}\right|.
\end{array}
\]
For $\ell\in \{1,\ldots, t_0\}$ and $k\in \{1, \ldots, \psi-1\}$, note that $\phi_{k,\ell} \le \alpha_{\ell}$ and $\mu_{k} \le {\alpha_0}$.
We have $ \sum_{j=1}^k \mbf{s}^{\pi(j)} = \sum_{\ell=1}^{t_0}\frac{\phi_{k,\ell}}{\alpha_\ell} \mbf{p}^{\ell} + \frac{\mu_{k}}{{\alpha_0}}\mbf{q}    \approx \sum_{\ell=1}^{t_0}\sum_{j=1}^{\phi_{k,\ell}}\mbf{C} \mbf{v}_{\ell,j}+ \sum_{j=1}^{\mu_{k}}\mbf{C}\widehat{\mbf{u}}_{j}$.
More precisely, by the Properties~\ref{technicalKeyProperties4} and~\ref{technicalKeyProperties9} there exists a constant 
\[
\omega_4\in\fpt(1)
\]
such that
\begin{equation}\label{eqDefineOmega2}
\left\|\left(\sum_{\ell=1}^{t_0}\sum_{j=1}^{\phi_{k,\ell}}\mbf{C} \mbf{v}_{\ell,j}+ \sum_{j=1}^{\mu_{k}}\mbf{C}\widehat{\mbf{u}}_{j}\right) - \left(\sum_{\ell=1}^{t_0}\frac{\phi_{k,\ell}}{\alpha_\ell} \mbf{p}^{\ell} + \frac{\mu_{k}}{{\alpha_0}}\mbf{q} \right) \right\|_{\infty}
\le 
\omega_4.
\end{equation}

Next we bound the number of distinct values that the integer vectors
\begin{equation}\label{eqBoundOffset}
\sum_{\ell=1}^{t_0}\sum_{j=1}^{\phi_{k,\ell}}\mbf{C} \mbf{v}_{\ell,j}+ \sum_{j=1}^{\mu_{k}}\mbf{C}\widehat{\mbf{u}}_j,
\end{equation}
can attain. 
We claim that they can attain at most
\begin{equation}\label{eqSetOmega3}
\omega_5 := 36 \left(s_0^{1/2} \left(\omega_3(\dim(\mcf{V})+1) + \omega_4+\sfrac{1}{2}\right) \right)^{\dim(\mcf{V})} \left(s_0^{1/2} \left(\omega_4+\sfrac{1}{2}\right) \right)^{s_0-\dim(\mcf{V})}
\end{equation}
distinct values.
Note that $\omega_5\in\fpt(n^{\min\{t_0+2,s_0\}})$.
By \eqref{eqRearrangeFractions} the partial sums $\sum_{j=1}^k \mbf{s}^j \in \mbb{R}^{s_0}$ lie in 
\[
\mcf{X} := \left(\omega_3 (\dim(\mcf{V})+1) \cdot [-1,1]^{s_0} \right)\cap \mcf{V}.
\]
By \eqref{eqDefineOmega2} the integer vectors in~\eqref{eqBoundOffset} deviate from $\sum_{j=1}^k \mbf{s}^j$ by an additive factor of $\omega_4$ in each of the $s_0$ components. 
Thus, the values in~\eqref{eqBoundOffset} lie in $\mcf{X}+\mcf{Y}$, where $\mcf{Y} := [-\omega_4,\omega_4]^{s_0}$.
To bound the number of integer vectors in $\mcf{X}+\mcf{Y}$, it suffices to bound the volume of $\mcf{X}+\mcf{Y}+[-\sfrac{1}{2},\sfrac{1}{2}]^{s_0}$.
We simplify computations by using the Euclidean ball $\mcf{E} := \{\mbf{x} \in \mbb{R}^{s_0}: \|\mbf{x}\|_2\le 1\}$ and using that 
\[
s_0^{-1/2}[-1,1]^{s_0} \subseteq \mcf{E} \subseteq (\mcf{E}\cap \mcf{V}) + (\mcf{E}\cap \mcf{V}^{\perp}).
\]
We obtain $\mcf{X}+\mcf{Y}+[-\sfrac{1}{2},\sfrac{1}{2}]^{s_0}$ is contained in the set
\[
\left[s_0^{1/2} \left(\omega_3(\dim(\mcf{V})+1) + \omega_4+\sfrac{1}{2}\right) \cdot \mcf{E} \right]\cap \mcf{V} + \left[s_0^{1/2} \left(\omega_4+\sfrac{1}{2}\right) \cdot \mcf{E} \right]\cap \mcf{V}^{\perp}.
\]
The volume of the Euclidean unit ball $\mcf{E}$ is upper bounded independently of $s_0$, e.g., by 6.
Hence, we can upper bound the volume of the latter set by $\omega_5$.

Finally, set
\begin{equation}\label{eqLargeConstant}%
\xi := (\omega_5 + t_0(t-s+2)+1){\omega_2} \omega_1t(2s\Delta+1)^s,
\end{equation}
which is in  $\fpt(n^{\min\{t_0+2,s_0\}})$.%
\footnote{One can argue that $\xi \in \mcf{O}(n^{\dim \mcf{V}} \cdot \gamma^{1+s_0+\dim\mcf{V}} \cdot (\Delta sts_0t_0)^{\mcf{O}((sts_0t_0)^3)})$.}
By Properties~\ref{technicalKeyProperties1} and \ref{technicalKeyProperties65} and then by \eqref{equ:boundVbyX}, we have
%
\begin{align*}
{\alpha_0}+ \sum_{\ell=1}^{t_0} \alpha_{\ell} &\ge \frac{\|\widehat{\mbf{u}}\|_\infty}{t(2s\Delta+1)^s}-1 + \frac{\|\widehat{\mbf{x}}\|_\infty}{{\omega_2}}-t_0(t-s+2) \\&\ge \frac{\|(\widehat{\mbf{x}},\widehat{\mbf{y}})\|_\infty}{{\omega_2} \omega_1t(2s\Delta+1)^s }-t_0(t-s+2)-1>  \omega_5.
\end{align*}
By the pigeonhole principle there exist $k < k'$ in $\{1, \ldots, \psi\}$ such that 
\[
\sum_{\ell=1}^{t_0}\sum_{j=\psi_{k,\ell}}^{\psi_{k',\ell}}\mbf{C} \mbf{v}_{\ell,j}+ \sum_{j=\mu_{k}}^{\mu_{k'}} \mbf{C}\widehat{\mbf{u}}_{j} = \mbf{0}. 
\]
%
By the construction of our decomposition we have
\[
\mbf{A}^i \left(\sum_{j=\psi_{k,\ell}}^{\psi_{k',\ell}} \mbf{v}^{i}_{\ell,j} +\sum_{j=\mu_{k}}^{\mu_{k'}} \widehat{\mbf{u}}^{i}_{j}\right) = -(\psi_{k',\ell} - \psi_{k, \ell}) \mbf{B}^i \mbf{h}^{\ell}
\]
for each $i \in \{ 1, \ldots, n\}$ and $\ell \in \{ 1, \ldots, t_0\}$.
Thus, the vector $({\mbf{x}}, {\mbf{y}}) \in \mbb{R}^{t_0} \times \mbb{R}^{nt}$ defined by
\[
{\mbf{x}}:= \sum_{\ell=1}^{t_0} (\psi_{k',\ell} - \psi_{k,\ell}) \mbf{h}^{\ell}
\]
and 
\[
{\mbf{y}}^i:= 
\sum_{\ell=1}^{t_0} \left(\sum_{j=\psi_{k,\ell}}^{\psi_{k',\ell}} \mbf{v}^{i}_{\ell,j} +\sum_{j=\mu_{k}}^{\mu_{k'}} \widehat{\mbf{u}}^{i}_{j}\right)
\]
for $i \in \{1, \ldots, n\}$, is a nonzero vector in $(\mbb{Z}^{t_0}_+ \times \mbb{Z}^{nt}_+) \cap  \ker \mbf{H}$ that satisfies $ ({\mbf{x}},{\mbf{y}}) \le (\widehat{\mbf{x}}, \widehat{\mbf{y}} )$.
This completes the proof of Theorem~\ref{thm4BlockVector}.
\hfill ~\qed
\end{proof}

\subsection{Decomposing $\widehat{\mbf{u}}$}

To prove the main statement of this subsection, Lemma~\ref{lemDecomposeU}, we will need the following auxiliary lemma,
that bounds the norm of minimal integer kernel vectors.
\begin{lemma}\label{lemSteinitzCycle}
Let $i \in\{1, \ldots, n\}$ and ${\mbf{w}} \in \mbb{R}^t_+ \cap \ker \mbf{A}^i$.
If $\|{\mbf{w}} \|_{1} > t(2s\Delta+1)^s$, then there exists a nonzero vector $\overline{\mbf{w}} \in \mbb{Z}^t_+ \cap \ker \mbf{A}^i$ such that $\overline{\mbf{w}} \le \mbf{w}$.
\end{lemma}
A proof can be found in~\cite[Section 3]{EW2018} by using $\mbf{w}$ in our notation in place of $\mbf{z}^* - \mbf{x}^*$ in their notation.

\begin{lemma}[Decomposing $\widehat{\mbf{u}}$]\label{lemDecomposeU}
There exists ${\alpha_0} \in \mbb{Z}_+$ such that
\[
\widehat{\mbf{u}} = \widehat{\mbf{u}}_{0} + \sum_{j=1}^{{\alpha_0}} \widehat{\mbf{u}}_j,
\]
where
\begin{enumerate}[label = \roman*)]
\item
$\widehat{\mbf{u}}_j \in \mbb{Z}^{nt}_+ \cap \ker \mbf{A}$ and $\|\widehat{\mbf{u}}_j \|_{1}\le t(2s\Delta+1)^s$ for $j \in \{1, \ldots, {\alpha_0}\}$.
\item
$\widehat{\mbf{u}}_{0} \in \mbb{R}^{nt}_+\cap \ker \mbf{A}$, $\|\widehat{\mbf{u}}_{0} \|_{1}\le n t(2s\Delta+1)^s$ and $\|\widehat{\mbf{u}}_{0} \|_{\infty}\le t(2s\Delta+1)^s$.
\item For $ k \in \{1, \ldots, {\alpha_0}\}$
\begin{equation}\label{eqRearrangeq}
\bigg\| \sum_{j=1}^k \mbf{C}\widehat{\mbf{u}}_j  - \frac{k}{{\alpha_0}} \sum_{j=1}^{\alpha_0}\mbf{C}\widehat{\mbf{u}}_j \bigg\|_{\infty} \le s_02\Delta t(2s\Delta+1)^s.
\end{equation}
\item The number $\alpha_0$ satisfies
\begin{equation}\label{eqBoundAlphaZero}
{\alpha_0} \ge \frac{\|\widehat{\mbf{u}}\|_\infty}{t(2s\Delta+1)^s}-1.
\end{equation}
\end{enumerate}
\end{lemma}
\begin{proof}
Fix $i \in \{1, \ldots, n\}$. 
Suppose $\|\widehat{\mbf{u}}^{i}\|_{1} > t(2s\Delta+1)^s$.
Apply Lemma~\ref{lemSteinitzCycle} $\beta_i$ times (for some $\beta_i \in \mbb{Z}$) to generate $\mbf{w}^i_1, \ldots, \mbf{w}^i_{\beta_i} \in \mbb{Z}^t_+$ until $\|\widehat{\mbf{u}}^{i} - \sum_{j=1}^{\beta_i}\mbf{w}^{i}_j \|_{1} \le t(2s\Delta+1)^s$.
Extend $\mbf{w}^{i}_j  \in \mbb{R}^t$ to a vector $\widehat{\mbf{u}}_j \in \mbb{R}^{nt}$ by setting $\widehat{\mbf{u}}_j^i := \mbf{w}^i_j$ and $\widehat{\mbf{u}}^{i'}_j = \mbf{0} \in \mbb{R}^t$ for each $i' \in \{1, \ldots, n \}\setminus \{i\}$.

Let $\beta_1, \dotsc, \beta_n$ be constructed as in the previous paragraph. 
The ${\alpha_0} := \sum_{i=1}^n \beta_i$ vectors defined in this way give us the desired $\widehat{\mbf{u}}_1, \ldots, \widehat{\mbf{u}}_{{\alpha_0}}$.
Set $\widehat{\mbf{u}}_{0}$ to be the remainder:
\[
\widehat{\mbf{u}}_{0} := \widehat{\mbf{u}} - \sum_{j=1}^{{\alpha_0}}\widehat{\mbf{u}}_j = \left(\widehat{\mbf{u}}^{1} - \sum_{j=1}^{{\alpha_0}}\widehat{\mbf{u}}^{1}_j,\ \ldots,\ \widehat{\mbf{u}}^{n} - \sum_{j=1}^{{\alpha_0}}\widehat{\mbf{u}}^{n}_j\right).
\]
Then {\it i)} and {\it ii)} follow.

To prove~\eqref{eqRearrangeq} we use the definition of $\Delta$ to derive
\[
\left\| \mbf{C} \widehat{\mbf{u}}_j\right\|_{\infty} \le \max\left\{\left\|\mbf{C}^i\right\|_{\infty}:\ i \in \{1, \ldots, n\}\right\} \left\|\widehat{\mbf{u}}_j\right\|_1 \le \Delta  t(2s\Delta+1)^s
\]
for each $j \in \{1, \ldots, {\alpha_0}\}$.
Applying Theorem~\ref{thmSteinitz} to the zero-sum sequence 
\[
\left(\mbf{C} \widehat{\mbf{u}}_k - \frac{1}{{\alpha_0}} \sum_{j=1}^{{\alpha_0}}  \mbf{C} \widehat{\mbf{u}}_j\right)_{k=1}^{{\alpha_0}}
\]
whose vectors have $\ell_{\infty}$-norm at most $2 \Delta  t(2s\Delta+1)^s$, we obtain~\eqref{eqRearrangeq}, possibly after relabelling.

It remains to prove~\eqref{eqBoundAlphaZero}.
Using parts {\it i)} and {\it ii)}, we have
\[
\left\|\widehat{\mbf{u}}\right\|_{\infty}
= \left\|\widehat{\mbf{u}}_0 +\sum_{j=1}^{\alpha_0}\widehat{\mbf{u}}_j\right\|_{\infty} \le (\alpha_0+1) t(2s\Delta+1)^s.
\]
Rearranging this inequality yields the result.
\hfill ~\qed
\end{proof}


\subsection{Decomposing $\widehat{\mbf{x}}$}


%
For each $i \in \{1, \ldots, n\}$, we define the set of feasible basis matrices of $\{\mbf{y} \in \mbb{R}^{t} : \mbf{A}^i \mbf{y} = - \mbf{B}^i \widehat{\mbf{x}}\}$:
\begin{equation}\label{definitionOfFeasibleBasis}
\mcf{B}^i := \left\{\mbf{D} \in \mbb{Z}^{s\times s}:\ 
\begin{array}{l}
\mbf{D} ~\text{is an invertible submatrix of}~\mbf{A}^i\\[.1 cm]
\text{and}~ -\mbf{D}^{-1} \mbf{B}^i \widehat{\mbf{x}} \ge \mbf{0}
\end{array}\right\}.
\end{equation}
%
%


\begin{lemma}[Decomposing $\widehat{\mbf{x}}$]\label{lemDecomposeX}
There exists $\lambda_1, \dotsc, \lambda_{t_0} \in \mbb{R}_+$ satisfying
\begin{equation}\label{eqDecomposeX}
\widehat{\mbf{x}} = \sum_{\ell=1}^{t_0} \lambda_\ell \mbf{h}^\ell,
\end{equation}
where
$\mbf{h}^1, \dotsc, \mbf{h}^{t_0} \in \mbb{Z}^{t_0}_+$ satisfy $\|\mbf{h}^{\ell}\|_{\infty} \le \omega_2$ for each $\ell \in \{1, \dotsc, t_0\}$ and
\[
-\mbf{D}^{-1} \mbf{B}^i \mbf{h}^{\ell} \ge \mbf{0}  ~\text{for every}~\ell \in \{1, \dotsc, t_0\},\ i \in \{1, \dotsc, n\} ~\text{and}~ D\in\mcf{B}^i
\]
%
and 
\(
\omega_2 \in \fpt(1).
\)
\end{lemma}

\begin{proof}
The vector $\widehat{\mbf{x}}$ is in the cone
\[
\mcf{K} := \left\{ \mbf{x} \in \mbb{R}^{t_0}_+ : -\mbf{D}^{-1} \mbf{B}^i \mbf{x} \ge \mbf{0}~~\forall ~i\in \{1,\ldots, n\}~\text{and}~\mbf{D} \in \mcf{B}^i\right\}.
\]
Let $\mcf{H} \subseteq \mcf{K}$ be a set of vectors that define the extreme rays of $\mcf{K}$.
Scale each $\mbf{h} \in \mcf{H}$ so that $\mbf{h} \in \gamma \mbb{Z}^{t_0}$, where
\begin{equation}\label{eqLCM}
\gamma := {\rm lcm} \left\{|\det \mbf{D}| : 
\begin{array}{l}\mbf{D} \in \mbb{Z}^{s\times s}  ~\text{is an invertible submatrix of}~ \mbf{A}^i\\
\text{for some}~i \in \{ 1, \ldots, n\}
\end{array}\right\}.
\end{equation}
In~\eqref{eqLCM} we use lcm to denote the least common multiple.
According to Hadamard's inequality $|\det \mbf{D}| \le \Delta^s s^{s/2}$ for each invertible submatrix $\mbf{D}$ of $\mbf{A}^i$.
Hence, 
\[
\gamma  \le \big(\Delta^s s^{s/2}\big)!
\]
Using standard techniques in polyhedral theory, see~\cite[\S 6.2]{GLS1988}, the number
\begin{equation}\label{eqDeltah}
\omega_2:=\max\left\{\|\mbf{h}\|_{\infty} : \mbf{h} \in \mcf{H}\right\}
\end{equation}
is in $\fpt(1)$.
By Carath\'eodory's Theorem we can write $\widehat{\mbf{x}}$ as a conic combination of $t_0$ vectors in $\mcf{H}$ as in~\eqref{eqDecomposeX}.
\hfill~\qed
\end{proof}

\subsection{Decomposing $\widehat{\mbf{v}}$}

The main result of this subsection is Lemma~\ref{lemDecomposeV} for which we need a number of auxiliary lemmas.
%
%
%
Throughout we use the notation from the decomposition of $\widehat{\mbf{x}}$ in Lemma~\ref{lemDecomposeX}.

\begin{lemma}\label{lemMinkSum}
\begin{enumerate}[label = \alph*)]\item
For $i \in \{ 1, \ldots, n\}$, we have
\[
\left\{ \mbf{y} \in \mbb{R}^{t}_+ : \mbf{B}^i\widehat{\mbf{x}} +\mbf{A}^i \mbf{y} = \mbf{0} \right\} =  \sum_{\ell=1}^{t_0} \left\{ \mbf{y} \in \mbb{R}^{t}_+ : \lambda_{\ell}\mbf{B}^i\mbf{h}^{\ell} + \mbf{A}^i \mbf{y} = \mbf{0}\right\},
\]
where the summation is a Minkowski Sum.
%
%

\item
It holds that
\begin{equation*}
\|(\widehat{\mbf{x}},\widehat{\mbf{y}})\|_{\infty}\le\|(\widehat{\mbf{x}},\widehat{\mbf{v}})\|_\infty+\|\widehat{\mbf{u}}\|_\infty \le \omega_1 \|\widehat{\mbf{x}}\|_\infty+\|\widehat{\mbf{u}}\|_\infty,
\end{equation*}
where
\[
\omega_1 := t_0 \Delta^{s} s^{(s+1)/2}.
\]
\end{enumerate}
\end{lemma}

%
\begin{proof}
For $\ell \in \{1, \ldots, t_0\}$, let 
\[
\mbf{y}^{\ell} \in \left\{ \mbf{y} \in \mbb{R}^{t}_+ : \lambda_{\ell}\mbf{B}^i\mbf{h}^{\ell} + \mbf{A}^i \mbf{y} = \mbf{0}\right\}.
\]
We have $\sum_{\ell=1}^{t_0} \mbf{y}^{\ell} \ge \mbf{0}$ and by Lemma~\ref{lemDecomposeX}
\[
\mbf{A}^i \left(\sum_{\ell=1}^{t_0} \mbf{y}^{\ell} \right) = -\mbf{B}^i \left(\sum_{\ell=1}^{t_0} \lambda_{\ell} \mbf{h}^{\ell} \right) = -\mbf{B}^i \widehat{\mbf{x}}.
\]
Thus, the `$\supseteq$' inclusion holds.

Let 
\[
\widetilde{\mbf{y}} \in \left\{ \mbf{y} \in \mbb{R}^{t}_+ : \mbf{A}^i \mbf{y} = -\mbf{B}^i \widehat{\mbf{x}}\right\} =: \mcf{Q}.
\]
For each $ \mbf{D} \in \mcf{B}^i$ (see Definition~\eqref{definitionOfFeasibleBasis}) and $\mbf{x} \in \mbb{R}^{t_0}$, we write $\mbf{y}_{(\mbf{D}, \mbf{x})}$ to denote the unique vector in $\mbb{R}^{t}$ whose support is contained in the columns of $\mbf{D}$ in $\mbf{A}^i$ and satisfies $\mbf{A}^i \mbf{y}_{(\mbf{D}, \mbf{x})} = - \mbf{B}^i \mbf{x}$.
By Carath\'eodory's Theorem, there exist vertices $\mbf{y}_{(\mbf{D}^1, \widehat{\mbf{x}})}, \ldots, \mbf{y}_{(\mbf{D}^{t+1}, \widehat{\mbf{x}})}$ of $\mcf{Q}$, numbers $\mu_1, \ldots, \mu_{t+1} \ge 0$ with $\sum_{k=1}^{t+1} \mu_k = 1$, and $\mbf{r} \in \mbb{R}_+^t \cap \ker \mbf{A}^i$ such that
\begin{equation}\label{eqConvexCombVertices}
\begin{array}{rcl}
\widetilde{\mbf{y}} &=&  \displaystyle\mbf{r} + \left(\sum_{k=1}^{t+1} \mu_k \mbf{y}_{(\mbf{D}^k, \widehat{\mbf{x}})}\right) \\[.5 cm]
& =& 
\displaystyle \sum_{\ell=1}^{t_0} \lambda_\ell \left(\frac{1}{\sum_{p=1}^{t_0} \lambda_p} \mbf{r} + \sum_{k=1}^{t+1} \mu_k   \mbf{y}_{(\mbf{D}^k, \mbf{h}^{\ell})}\right),
\end{array}
\end{equation}
where the second equation follows from~\eqref{eqDecomposeX}.
Given that $\mbf{D}^1, \ldots, \mbf{D}^{t+1} \in \mcf{B}^i$, we have $\mbf{y}_{(\mbf{D}^k, \mbf{h}^{\ell})}\ge \mbf{0}$ for each $k \in \{1, \ldots, t+1\}$ and $\ell \in \{1, \ldots, t_0\}$.
Thus, 
\[
\mbf{y}^{i}_{\ell} :=  \lambda_{\ell} \left(\frac{1}{\sum_{p=1}^{t_0} \lambda_p} \mbf{r} + \sum_{k=1}^{t+1} \mu_k  \mbf{y}_{(\mbf{D}^k, \mbf{h}^{\ell})}\right) \ge \mbf{0}
\]
and $\mbf{A}^i \mbf{y}^{i}_{\ell} = \lambda_{\ell} \sum_{k=1}^{t+1} \mu_k \mbf{A}^i \mbf{y}_{(\mbf{D}^k, \mbf{h}^{\ell})} = -\lambda_{\ell}\mbf{B}^i \mbf{h}^{\ell}$.
Thus, the `$\subseteq$' inclusion holds. 

Now we prove Part b). 
It suffices to prove $\|\widehat{\mbf{y}}^{i}-\widehat{\mbf{u}}^{i}\|_{\infty} \le \omega_1  \|\widehat{\mbf{x}}\|_{\infty}$ for each $i \in \{1, \ldots, n\}$.
Using the notation in the previous paragraph we write $\widetilde{\mbf{y}} = \widehat{\mbf{y}}^{i}-\widehat{\mbf{u}}^{i}$.
For this choice of $\widetilde{\mbf{y}}$, assumption \eqref{assumeNoMoreu} then implies that $\mbf{r}=\mbf{0}$ in~\eqref{eqConvexCombVertices}.
Hence, $\widehat{\mbf{y}}^{i}-\widehat{\mbf{u}}^{i} = \sum_{k=1}^{t+1} \mu_k (\mbf{D}^k)^{-1} \mbf{B}^i \widehat{\mbf{x}}$, where $\mu_1, \ldots, \mu_{t+1}$ are convex multipliers. 
By Hadamard's inequality and Cramer's Rule we have
\[
\|\widehat{\mbf{v}}^{i}\|_{\infty} \le \  t_0  \max_k\{\|(\mbf{D}^k)^{-1} \mbf{B}^i\|_\infty\}\|\widehat{\mbf{x}}\|_{\infty}
\le t_0  \Delta^{s} s^{(s+1)/2}  \|\widehat{\mbf{x}}\|_{\infty}.
\]
\hfill~\qed
\end{proof}


\begin{lemma}\label{lem:localCycle}
Let $i \in \{1, \ldots, n\}$ and $\ell \in \{1, \ldots, t_0\}$.
Let $\mbf{w} \in \mbb{R}^t_+$ and $\beta \in \mbb{R}_+$ such that $(\beta\mbf{h}^{\ell}, \mbf{w}) \in \ker \begin{bmatrix}\mbf{B}^i & \mbf{A}^i\end{bmatrix}$. 
Then
\begin{enumerate}[label = \roman*)]
\item If $\beta \ge  t-s+1$, then there exists a vector $\overline{\mbf{w}} \in\mbb{Z}_+^t$ such that $\overline{\mbf{w}} \le\mbf{w} $, $\|\overline{\mbf{w}}\|_1 \le \Delta^{s+1}s^s t_0 {\omega_2}$, and $(\mbf{h}^{\ell}, \overline{\mbf{w}}) \in \ker \begin{bmatrix}\mbf{B}^i & \mbf{A}^i\end{bmatrix}$.
\item If $\mbf{w} \le \mbf{v}^i_{\ell}$, then $\|\mbf{w}\|_1 \le \beta\Delta^{s+1}s^s t_0 {\omega_2}$.

\end{enumerate}
\end{lemma}

\begin{proof}
Define the polyhedron
\[
\mcf{Q} := \left\{ \mbf{y} \in \mbb{R}^{t}_+ : \mbf{B}^i\mbf{h}^{\ell} +\mbf{A}^i \mbf{y} = \mbf{0} \right\}
\]
when $\beta > 0$ and 
\[
\mcf{Q} := \mbb{R}^{t}_+ \cap \ker \mbf{A}^i 
\]
when $\beta =0$.
The polyhedron $\mcf{Q} $ is nonempty because it contains $\sfrac{1}{\beta} \cdot  \mbf{w} $ when $\beta> 0$ and $\mbf{0}$ when $\beta = 0$.
Set $\mbf{v}^{\mbf{D}} := -\mbf{D}^{-1}\mbf{B}^i \mbf{h}^{\ell}$ for each $\mbf{D}  \in \mbb{Z}^{s\times s}$ that is an invertible submatrix of $\mbf{A}^i$ and a feasible basis matrix for $\mcf{Q}$; thus, $\mbf{v}^{\mbf{D}} \ge \mbf{0}$.
The points $\mbf{v}^{\mbf{D}}$ are the vertices of $\mcf{Q}$, and they are integer-valued because we have scaled $\mbf{h}^{\ell}$ to lie in $\gamma \mbb{Z}^{t_0}$ and $\gamma$ satisfies~\eqref{eqLCM}.
For each $\mbf{D}$, we use Hadamard's inequality and the definition of ${\omega_2}$ in~\eqref{eqDeltah} to conclude
\begin{equation}\label{eqHadamard}
\left\|\mbf{v}^{\mbf{D}}\right\|_1 
\le 
\left\| -\mbf{D}^{-1}\right\|_{\infty} \left\|\mbf{B}^i\right\|_{\infty} \left\|\mbf{h}^{\ell}\right\|_1 \le \Delta^{s+1}s^st_0 {\omega_2}.
\end{equation}

The vector $ \mbf{w}$ is contained in
\[
\begin{array}{rcl}
&&\displaystyle\left\{ \mbf{y} \in \mbb{R}^{t}_+ : \beta\mbf{B}^i\mbf{h}^{\ell}+\mbf{A}^i \mbf{y}= \mbf{0} \right\}\\[.2 cm]
&=&\displaystyle\left\{ \mbf{r} \in \mbb{R}^{t}_+ : \mbf{A}^i \mbf{r}= \mbf{0} \right\} + \conv\left\{\beta\mbf{v}^{\mbf{D}} : \mbf{D} \text{ is a feasible basis matrix for $\mcf{Q}$}\right\}.
\end{array}
\]
Hence, there exists $\mbf{r} \in \mbb{R}^t_+\cap \ker\mbf{A}^i$ and coefficients $\tau_{\mbf{D}} \in \mbb{R}_+$ for each feasible basis matrix for $\mcf{Q}$ such that $\mbf{w} =\mbf{r} + \sum_{\mbf{D}}\tau_{\mbf{D}} \beta \mbf{v}^{\mbf{D}}$ and $\sum_{\mbf{D}}\tau_{\mbf{D}}=1$.
By Carath\'{e}odory's Theorem, we can choose the coefficients $\tau_{\mbf{D}}$ such that at most $t-s+1$ are nonzero.
Thus, there exists at least one $\overline{\mbf{D}}$ such that $\tau_{\overline{\mbf{D}}} \ge \sfrac{1}{(t-s+1)}$.

If $\beta\ge t-s+1$, then $\tau_{\overline{\mbf{D}}} \beta \ge 1$.
Set $\overline{\mbf{w}}:= \mbf{v}^{\overline{\mbf{D}}}$.
The result follows from~\eqref{eqHadamard}.
If $\mbf{w} \le \mbf{v}^i_{\ell}$, then $\mbf{r} = \mbf{0}$ by~\eqref{assumeNoMoreu}.
We use~\eqref{eqHadamard} to conclude
\[
\left\|\mbf{w}\right\|_1 =\left\|\sum_{\mbf{D}}\tau_{\mbf{D}} \beta \mbf{v}^{\mbf{D}}\right\|_1 \le \beta \Delta^{s+1}s^st_0 {\omega_2}.
\]
\hfill ~\qed
\end{proof}

The next lemma is where we apply the colorful Steinitz Lemma. 

\begin{lemma}[Decomposing $\widehat{\mbf{v}}$]\label{lemDecomposeV}
We can write $\widehat{\mbf{v}}   = \sum_{\ell=1}^{t_0} \mbf{v}_{\ell} $,
such that  $\mbf{v}_{\ell} \in \mbb{R}^{nt}$ and 
%
\[
\lambda_{\ell}\mbf{B}^i \mbf{h}^{\ell}
 + \mbf{A}^i \mbf{v}^i_{\ell} = \mbf{0}
 \]
For each $\ell \in \{1, \ldots, t_0\}$, the vector $\mbf{v}_{\ell} \in \mbb{R}^{nt}_+$ can be written as
\[
\mbf{v}_{\ell} = \sum_{j=0}^{\alpha_{\ell}} \mbf{v}_{\ell,j},
\]
where 
\begin{enumerate}[label = \roman*)]
\item
\begin{equation*}
\sum_{\ell=1}^{t_0} \alpha_{\ell}  \ge \frac{\|\widehat{\mbf{x}}\|_\infty}{{\omega_2}}-t_0(t-s+2) .
\end{equation*}
\item $\mbf{v}_{\ell, j} \in \mbb{Z}^{nt}_+$,  $(\mbf{h}^{\ell}, \mbf{v}_{\ell, j}) \in \ker \begin{bmatrix}\mbf{B} & \mbf{A}\end{bmatrix}$ and $\|\mbf{v}^{i}_{\ell, j}\|_1 \le  \Delta^{s+1}s^s t_0 {\omega_2}$ for each $i \in \{1,\ldots, n\}$ and $j \in \{1, \ldots, \alpha_{\ell}\}$.

\smallskip
\item$ \mbf{v}_{\ell,0} \in \mbb{R}^{nt}_+$ satisfies $\mbf{A} \mbf{v}_{\ell,0} \in \mbb{Z}^{ns}$ and $\|\mbf{v}^{i}_{\ell,0} \|_{1} \le(t-s+1)\Delta^{s+1}s^s t_0 {\omega_2}$ for each $i \in \{1,\ldots, n\}$.
\smallskip
\item If $\alpha_{\ell} \ge 1$, then 
\begin{equation*}
\bigg\|\sum_{j=1}^k\mbf{C}\mbf{v}_{\ell,j} - \frac{k}{\alpha_{\ell}}\sum_{j=1}^{\alpha_{\ell}} \mbf{C}\mbf{v}_{\ell,j} \bigg\|_{\infty} \le  40 s_0^5  \Delta^{s+2}s^s t_0 {\omega_2}
\end{equation*}
for each $k \in \{1, \ldots, \alpha_{\ell}\}$.
\end{enumerate}
\end{lemma}

\begin{proof}
Since $(\widehat{\mbf{x}}, \widehat{\mbf{v}}) \in \ker \begin{bmatrix}\mbf{B} &\mbf{A} \end{bmatrix}$, we can apply Lemma~\ref{lemMinkSum} to $\widehat{\mbf{v}}^i$ for each $i \in \{1, \dotsc, n\}$ to obtain
\begin{equation*}
\widehat{\mbf{v}}   = \sum_{\ell=1}^{t_0} \mbf{v}_{\ell} =  \left(\sum_{\ell=1}^{t_0} \mbf{v}^{1}_{\ell},\ \ldots,\ \sum_{\ell=1}^{t_0} \mbf{v}^{n}_{\ell}\right),
\end{equation*}
where $\mbf{v}_{\ell} \in \mbb{R}^{nt}$ and 
%
\[
\lambda_{\ell}\mbf{B}^i \mbf{h}^{\ell}
 + \mbf{A}^i \mbf{v}^i_{\ell} = \mbf{0}
 \]
for each $\ell \in \{1, \ldots, t_0\}$ and $i \in \{1, \ldots, n\}$.

Let $\ell \in \{1, \ldots, t_0\}$ and assume $\lambda_{\ell}$ is large.
The vectors $\mbf{w}, \overline{\mbf{w}}$ in Lemma~\ref{lem:localCycle} satisfy $\mbf{A}^i({\mbf{w}} - \overline{\mbf{w}}) = -(\lambda_{\ell}-1) \mbf{B}^i \mbf{h}^{\ell}$, and we can apply the lemma repeatedly. 
The number of times we can apply the lemma is
\begin{equation*}
\alpha_{\ell} := 
\begin{cases} 
\floor{\lambda_\ell - (t-s+1)} & \text{if} ~\lambda_\ell \ge t-s+1\\[.15 cm]
0& \text{if}~ \lambda_\ell \le t-s.
\end{cases}
\end{equation*}

From~\eqref{eqDecomposeX} it follows that $\|\widehat{\mbf{x}}\|_{\infty} \le  \sum_{\ell = 1}^{t_0} \lambda_{\ell} {\omega_2}$.
Thus,
\begin{equation*}
\sum_{\ell=1}^{t_0} \alpha_{\ell} \ge \sum_{\ell = 1}^{t_0} \big(\lambda_{\ell} - (t-s+2)\big) \ge \frac{\|\widehat{\mbf{x}}\|_\infty}{{\omega_2}}-t_0(t-s+2) .
\end{equation*}
We apply Lemma~\ref{lem:localCycle} $\alpha_{\ell}$ times to $\mbf{w} = \mbf{v}^i_{\ell}$.
After this, the vector $\mbf{v}^{i}_{\ell}$ can be written as 
\[
\mbf{v}^{i}_{\ell} = \mbf{v}^{i}_{\ell,0} + \sum_{j=1}^{\alpha_{\ell}} \mbf{v}^{i}_{\ell,j},
\]
where $ \mbf{v}^{i}_{\ell,0} \in \mbb{R}^t_+$, and $\mbf{v}^i_{\ell, j} \in \mbb{Z}^t_+$ and $(\mbf{h}^{\ell}, \mbf{v}^{i}_{\ell, j}) \in\ker \begin{bmatrix} \mbf{B}^i & \mbf{A}^i \end{bmatrix}$ for each $j \in \{1, \ldots, \alpha_{\ell}\}$.
The fact that $\|\mbf{v}^{i}_{\ell, j}\|_1 \le \Delta^{s+1}s^s t_0 {\omega_2}$ follows from Lemma~\ref{lem:localCycle}.
Similarly, $\|\mbf{v}^{i}_{\ell, 0}\|_1 \le (t-s+1) \Delta^{s+1}s^st_0\omega_2$ from Lemma~\ref{lem:localCycle}.
Furthermore,
\[
\mbf{A}^i \mbf{v}^{i}_{\ell,0} = -(\lambda_{\ell} - \alpha_{\ell}) \mbf{B}^i \mbf{h}^{\ell} = \mbf{A}^i \mbf{v}^i_{\ell}  - \mbf{A}^i\left(\sum_{j=1}^{\alpha_{\ell}} \mbf{v}^{i}_{\ell,j}\right)  \in \mbb{Z}^{s}
\]

It remains to show {\it iii).}
We use Corollary~\ref{corAffineVarSteinitz} of the colorful Steinitz Lemma (Theorem~\ref{thmVarSteinitz}).
For $i \in \{ 1, \ldots, n\}$ and $j \in\{1, \ldots, \alpha_\ell\}$, we have 
\[
\left\|\mbf{C}^i \mbf{v}^{i}_{\ell,j}\right\|_{\infty} \le \left\|\mbf{C}^i\right\|_{\infty} \left\|\mbf{v}^{i}_{\ell,j}\right\|_{1} \le  \Delta^{s+2}s^s t_0 {\omega_2}.
\]
Given that $\alpha_{\ell}\ge 1$, we may apply Corollary~\ref{corAffineVarSteinitz} in dimension $s_0$ with the $n$ sequences 
\[
\big(\mbf{C}^1 \mbf{v}^{1}_{\ell,j}\big)_{j=1}^{\alpha_\ell},~ \ldots ,~ \big(\mbf{C}^n \mbf{v}^{n}_{\ell,j}\big)_{j=1}^{\alpha_\ell}
\]
 to find permutations $\pi_{1,\ell}, \ldots, \pi_{n,\ell} \in \mcf{S}^{\alpha_{\ell}}$ such that 
\[
\bigg\|\sum_{j=1}^k \sum_{i=1}^n \mbf{C}^i \mbf{v}^{i}_{\ell,\pi_{i,\ell}(j)} - \frac{k}{\alpha_{\ell}}\sum_{j=1}^{\alpha_{\ell}} \sum_{i=1}^n\mbf{C}^i \mbf{v}^{i}_{\ell, j} \bigg\|_{\infty} \le 40 s_0^5   \Delta^{s+2}s^s t_0 {\omega_2}
\]
for each $k \in \{ 1, \ldots, \alpha_{\ell}\}$.
By relabeling, we may assume each $\pi_{i, \ell}$ is the identity permutation.
This completes the proof. 
\hfill ~\qed
\end{proof}


\begin{acknowledgements}
We are grateful to Fritz Eisenbrand for several helpful discussions. 
The second author was supported by a Natural Sciences and Engineering Research Council of Canada (NSERC) Discovery Grant [RGPIN-2021-02475].
The third author was supported by the Einstein Foundation Berlin. 
We thank the reviewers whose comments improved the readability of the paper and simplified some proofs.
\end{acknowledgements}

\bibliographystyle{plain}
\bibliography{references.bib}

\end{document}